\documentclass[12pt]{amsart}
\usepackage{amsmath}
\usepackage{amsxtra}
\usepackage{amscd}
\usepackage{amsthm}
\usepackage{amsfonts}
\usepackage{amssymb}
\usepackage{eucal}
\usepackage{epsfig}
\usepackage{bm}
\usepackage{graphics}
\usepackage{graphicx}
\usepackage{accents}
\usepackage{dynkin-diagrams}
\usepackage{comment}
\usepackage{enumitem}
\usepackage{mathrsfs}
\usepackage{mathtools}
\usepackage{tikz-cd}
\usepackage{mathtools}
\usepackage{pgf}
\usetikzlibrary{arrows, matrix}
\usepackage{ytableau}
\usepackage{here} 
\usepackage{simplewick}
\usepackage[hidelinks]{hyperref}
\usepackage{cancel}

\numberwithin{equation}{section}
\newtheorem{thm}{Theorem}[section]
\newtheorem{prop}[thm]{Proposition}
\newtheorem{lem}[thm]{Lemma}

\AtBeginEnvironment{rem}{%
  \pushQED{\qed}%
}
\AtEndEnvironment{rem}{\popQED\endexample}

\newtheorem{cor}[thm]{Corollary}

\theoremstyle{definition}
\newtheorem{example}[thm]{Example}
\AtBeginEnvironment{example}{%
  \pushQED{\qed}%
}
\AtEndEnvironment{example}{\popQED\endexample}

\newcommand{\C}{{\mathbb C}}
\newcommand{\Z}{{\mathbb Z}}

\newcommand{\R}{{\mathbb R}}

\newcommand{\bs}{\boldsymbol}

\newcommand{\sL}{\mathfrak{sl}}

\newcommand{\n}{\mathfrak n}
\newcommand{\h}{\mathfrak h}
\newcommand{\g}{\mathfrak g}
\newcommand{\coch}{\mbox{coch}}

\newcommand{\bomega}{\boldsymbol \omega}

\newcommand{\la}{\lambda}

\newcommand{\ch}{{\rm ch}}

\newcommand{\Ker}{\mathop{\rm Ker}}

\newcommand{\on}{\operatorname}

\newcommand{\al}{\alpha}

\newcommand{\ev}{\mathrm{ev}}

\newcommand{\be}{\begin{equation*}}
\newcommand{\ee}{\end{equation*}}

\begin{document}

\author{Maico Freitas and Evgeny Mukhin}
\address{MF: Department of Mathematical Sciences,
Indiana University Indianapolis,
402 N. Blackford St., LD 270, 
Indianapolis, IN 46202, USA}
\email{mgfreita@iu.edu}
\address{EM: Department of Mathematical Sciences,
Indiana University Indianapolis,
402 N. Blackford St., LD 270, 
Indianapolis, IN 46202, USA}
\email{emukhin@iu.edu}

\title{The deformed Tanisaki-Garsia-Procesi modules}

\maketitle

\begin{abstract}  The polynomial ideals  studied by A. Garsia and C. Procesi 
play an important role in the theory of Kostka polynomials. We give multiparameter flat deformations of these ideals and define an action of the extended affine symmetric group on the corresponding quotient algebras multiplied by the sign representation. We show that the images of these modules under the affine Schur-Weyl duality are dual to the local Weyl modules for the loop algebra $\sL_{n+1}[t^{\pm 1}].$  
\end{abstract}

\section{Introduction}
In \cite{DP},  C. De Concini and C. Procesi studied  a family of remarkable algebras $R(\la)$  labeled by partitions $\la$ of $d$. The algebra $R(\la)$ is the algebra of regular functions on the intersection of the set of diagonal $d\times d$ complex matrices with the closure of the set of nilpotent matrices whose  Jordan block sizes are parts of the partition $\la$. In \cite{T}, T. Tanisaki described  the algebra $R(\la)$ as a quotient of $\C[t_1,\dots,t_d]$ by an explicit graded symmetric ideal $\langle {\mathscr C}_\la \rangle$. In \cite{GP},  A. Garsia and C. Procesi proved that, as a graded module for the symmetric group $S_d$, the algebra $R(\la)$ has a remarkable decomposition:
$$\setlength\abovedisplayskip{8pt}
\setlength\belowdisplayskip{6pt}
R(\lambda) \cong \sum_{\mu \,\vdash d }\widetilde K_{\mu, \,\lambda^t}(q)L(\mu),$$
where $L(\mu)$ is the irreducible $S_d$-module corresponding to the partition $\mu$ of $d$, and  $\widetilde K_{\mu,\la^t}(q)$ is the celebrated modified Kostka polynomial.

On the other hand, there is a famous family of $\mathfrak{sl}_{n+1}[t^{\pm 1}]$ finite-dimensional cyclic modules defined in \cite{CP}, called the local Weyl modules $W_a(\la)$, which depend on a partition with at most $n$ parts $\la$ and non-zero complex parameters $\bs a=(a_1,\dots,a_{\la_1})$, see \eqref{weylmodule}. It is well-known that the local Weyl modules $W_{\bs a}(\la)$  have a similar decomposition as  $\mathfrak{sl}_{n+1}$-modules:
$$
W_{\bs a}(\lambda) \cong \sum_{\mu \,\vdash |\la| }\widetilde K_{\mu, \,\lambda}(1)V(\mu),
$$
where $V(\mu)$ is the irreducible $\mathfrak{sl}_{n+1}$-module of highest weight corresponding to $\mu$.  

In this paper, we show that there is a natural explanation for this coincidence of decompositions. The representation theories of the symmetric group $S_d$ and of the Lie algebra $\mathfrak{sl}_{n+1}$ are connected by the classic Schur-Weyl duality. There is an extension of this duality to the affine setting, known as the affine Schur-Weyl duality, relating the representations of the extended affine symmetric group  $\widetilde S_d$ and of the loop algebra $\mathfrak{sl}_{n+1}[t^{\pm 1}]$, see \cite{CP96}, \cite{Y}, \cite{FKM}. We give a deformation $\langle {\mathscr C}_\la^{\,\bs a}\rangle$ of the ideal $\langle {\mathscr C}_\la \rangle$ depending on  a sequence of parameters $\bs a$ and show that  $R_{\bs a}(\la)=\C[t_1,\dots,t_d]/\langle {\mathscr C}^{\, \bs a}_\la \rangle$ multiplied by the one-dimensional sign $S_d$ representation $L(1,\dots,1)$ is an  $\widetilde S_d$-module whose image with respect to the affine Schur-Weyl duality is dual to the local Weyl module $W_{\bs a}(\la)$, see Theorem \ref{dualweyl}.


The ideal $\langle {\mathscr C}_\la^{\,\bs a}\rangle$  is described explicitly by the requirement that the ratio
$$
 \frac{\prod_{j\in J} (x+t_j)}{ \prod_{i=1}^{m_\la(n)}(x+\bs a^{(n)}_i)}
$$
is a polynomial in $x$ for all $n, J$. Here, $J\subseteq \{1,\dots,|\la|\}$ is any subset of cardinality $|\la|-n$,  $m_\la(n)=\la_{n+1}+\la_{n+2}+\dots$, and  $\bs a^{(n)}=(a_1,a_2,\dots,a_{\la_{n+1}},a_1,\dots,a_{\la_{n+2}},\dots)$. In particular, we prove that this deformation is flat:  the $S_d$-module structure of  $R_{\bs a}(\la)$  does not depend on $\bs a$ and, in particular,  $\dim R_{\bs a}(\la)$ is the same for all $\bs a$, see Theorem \ref{flat thm}. We call the algebra $R_{\bs a}(\la)$ the deformed Tanisaki-Garsia-Procesi \cancel{(GPT)} (TGP) algebra.

The affine symmetric group is the semidirect product $\widetilde S_d = S_d \ltimes \mathbb Z^d$. If we identify the group algebra $\mathbb C \mathbb Z^d  = \mathbb C[t_1^{\pm },\dots, t_d^{\pm 1}]$, the TGP algebra is naturally a cyclic $\widetilde S_d$ module whose cyclic vector $1$ is fixed by the symmetric group $S_d$. Analogously, the dual to the local Weyl module $W_{\bs a}(\la)$ is also a cyclic $\mathfrak{sl}_{n+1}[t^{\pm 1}]$-module generated from the $\mathfrak{sl}_{n+1}$-module $V(1,1,\dots,1)$. If all $a_i$ are distinct, the local Weyl module $W_{\bs a}(\la)$ is irreducible and isomorphic to a tensor product of evaluation modules corresponding to the columns of $\la$ and evaluation parameters $a_i$. In particular, it is easy to find its preimage with respect to the affine Schur-Weyl duality and describe it as a quotient of $\mathbb C[t_1,\dots,t_d]$. Then, one has to take the limit as various $a_i$ may go to the same limiting value. Such a limit is subtle and is similar to the construction of the fusion tensor products, see \cite{CL}. We illustrate the process of taking the limit by an example, see Section \ref{amendedlocal}.

\medskip

Using the results of this paper one can describe the preimage of the graded local Weyl module of the current $\sL_{n+1}$ algebra as the dual to the non-deformed TGP module $R(\la)$ multiplied by the sign representation of $S_d$. We were informed by A. Moura that this result is known to him. For the case of one row partitions, the similarity between graded local Weyl modules and TGP algebras was observed in \cite{K}.

\medskip

The paper is constructed as follows. In Section \ref{prel sec}, we define the main objects of our study: the extended symmetric group $S_d$, the loop algebra $\mathfrak{sl}_{n+1}[t^{\pm 1}]$ and give some facts about their representations. In Section \ref{SW}, we recall the affine Schur-Weyl duality and show it is compatible with taking duals. In Section \ref{TGP sec}, we define the Tanisaki-Garsia-Procesi algebras, recall their characters in terms of the modified Kostka polynomials, give a deformation of the TGP algebras, and establish some simple properties. In Section \ref{main sec}, we prove our main results: that the deformations of the TGP algebras are flat, Theorem \ref{flat thm} and that the images of the affine Schur-Weyl functor of the deformed TGP modules are dual to local Weyl modules, Corollary \ref{dual}. We conclude with an example in Section \ref{example sec}.
 
\section{Preliminaries}\label{prel sec}

\subsection{ Partitions and Young tableaux} 

Let $d\in \mathbb Z_{\geq  0}$. A partition of $d$ is a finite sequence of non-negative integers of the form $\lambda =(\lambda_1,\dots,\lambda_p)$ such that $\sum_{i=1}^p \lambda_i = d$ and $\lambda_i\geq \lambda_{i+1} $. The integers $\lambda_i\geq 0$ occurring in $\la$ are usually referred to as the parts of $\la$.  We identify partitions $\lambda $ and $\mu$ which differ only by parts that are zero.

Denote by $\mathscr P_d$ the  set of partitions of $d$.
For a partition $\lambda$ of $d$, we often write $\lambda\vdash d$ to indicate that $\lambda \in \mathscr P_d$.
If $d>0$ and the  non-zero distinct parts of $\la$ are $d_1>d_2>\dots>d_s$ with multiplicities $m_1,\dots, m_s$, sometimes we simplify notation and write $\lambda = \big(d_1^{\, \{ m_1\} },\dots, d_s^{\,\{ m_s \} }\big) $. The number $\ell(\lambda)=\sum_{i=1}^s m_i$ of non-zero parts of $\lambda$ is called the length of the partition $\lambda$. 

In what follows, we will have various objects depending on a partition $\la$, such as  $L(\la),$ $R(\la)$, etc. If $\la=\big(d_1^{\, \{ m_1\} },\dots, d_s^{\,\{ m_s \} }\big)$, we omit the second set of parenthesis and simply write  $L(d_1^{\, \{ m_1\} },\dots, d_s^{\,\{ m_s \} })$,  $R(d_1^{\, \{ m_1\} },\dots, d_s^{\,\{ m_s \} })$, etc, instead of  
$L((d_1^{\, \{ m_1\} },\dots, d_s^{\,\{ m_s \} }))$, $R((d_1^{\, \{ m_1\} },\dots, d_s^{\,\{ m_s \} }))$, etc.

Let $\preceq $ be the dominance partial order on $\mathscr P_d$ given by $\lambda\preceq \mu$  if and only if 
\begin{equation*} 
\setlength\abovedisplayskip{8pt}
\setlength\belowdisplayskip{8pt}
    \lambda_1+\cdots + \lambda_k \leq  \mu_1+\cdots +\mu_k \qquad \mbox{for all} \ \ 1\leq k \leq d.
\end{equation*}

For $\lambda \vdash d,$ the Young diagram of $\lambda$ is the left-aligned array of $d$ boxes consisting of $\ell(\lambda)$ rows with $\lambda_i$ boxes in the $i$-th row. We use the English convention in which the first row is the top row. The transposed partition of $\lambda$ is the partition $\lambda^t$  whose parts are given by  $\lambda^t_{\,i} = \#\{j: \lambda_j\geq i\}$. The Young diagram of  $\lambda^t$ is obtained from the Young diagram of $\la$ by switching the rows and columns.

Given $\lambda \vdash d$, a Young tableau of shape $\lambda$ consists of a Young diagram of $\lambda$ in which each box is filled with an element of the set $\{1,\dots, d\}.$ A Young tableau is said to be  semi-standard if the fillings weakly increase  in rows from left to right, and strictly increase in columns from top to bottom. 
A semi-standard Young tableau is said to be standard if the fillings are all distinct. Denote by $\mbox{SSYT}(\lambda)$ the set of semi-standard Young tableaux of shape $\lambda$.

For $\lambda, \mu \vdash d$, we say that $T\in \mbox{SSYT}(\lambda)$ has content $\mu$ if $T$ has $\mu_i$ fillings equal to $i$. Denote by $\mbox{SSYT}(\lambda, \mu)$  the set of all semi-standard Young tableaux of shape $\lambda$ and content $\mu$.  The cardinality of $\mbox{SSYT}(\lambda,\mu)$ is the Kostka number  $K_{\la, \,\mu}$. 
In particular, we have $K_{(d),\, \lambda} = K_{\lambda,\, \lambda} = 1$ for all $\lambda \vdash d$ and $K_{\lambda,\, \mu} > 0$ if and only if $\mu\trianglelefteq \lambda$.

\subsection{Words, subwords and cocharges}\label{sec cocharge}
A word $w$ is a finite sequence of positive integers. We refer to any subsequence of $w$ as a subword of $w$. We say that a word $w$ has content $\lambda=(\lambda_1,\lambda_2,\dots)$ if $w$ has the integer $i$ occurring $\lambda_i$ times.
We say that a word $w$ is standard if it has content $\lambda=(1,1,\dots,1,0,0,\dots)$.

A word $w$ whose content $\lambda$ is a partition may be decomposed into  standard subwords $w_1,w_2,\dots,w_{\lambda_1}$ as follows. To create $w_1$, we start with the rightmost $1$ in $w$. Then, we find the rightmost $2$ to the left of that $1$. If no $2$ lies to the left of our $1$, we take the rightmost $2$ in $w$. If there is no $2$ at all, we end the selection of $w_1$. If $2$ was selected, we take the rightmost $3$ to the left of that $2$. If there is no $3$ to the left, we take the rightmost $3$ in $w$ or end the selection if there is no $3$ at all. We continue in this way with the next integers $4,5,$ etc, until we end the process.

After we finish the selection, we let $w_1$ be the subword of $w$ consisting of the selected numbers written in the same order as in $w$.  After $w_1$ is formed, we delete the integers participating in $w_1$ from $w$ and find $w_2$ in the same way from the remaining word.  We continue this process $\lambda_1$ times.

Let $w$ be a standard word. The cocharge $\coch(i)$ of an element $i$ of $w$ is defined recursively as follows. Set $\coch(1) = 0$ and let
\be
\coch(i)=\begin{cases}
  \coch(i-1) + 1,  \quad {\rm if}\  i\ {\rm lies\ to\ the\ left\ of\ } (i-1)\ {\rm in}\ w, \\
  \coch(i-1), \quad {\rm if}\  i\ {\rm lies\ to\ the\ right\ of\ } (i-1)\ {\rm in}\ w.
\end{cases}
\ee
We define the cocharge $\coch(w)$ of the standard word  $w$ as the sum $$\coch(w)=
\sum_{i\in  w}\coch(i).$$

For an arbitrary word $w$ whose content is a partition $\lambda$, the cocharge of $w$ is given by the sum of the cocharges of its standard subwords $$\coch(w)=
\sum_{i=1}^{\lambda_1}\coch(w_i).$$

\begin{example} \label{cocharge example}
    Let 
    \be
w =  4 \, 2 \, 2 \, 3 \, 1 \, 1 \, 1 \, 2 \, 3. 
    \ee
Then,  $w$  has content  $\lambda = (3,3,2,1)$ and standard subwords
\be
w_1= 4\, 2\, 1\, 3 , \qquad w_2= 2\,3\,1,  \qquad w_3= 1\,2.
\ee

The cocharge of the elements of $w_1$ are
\be
\coch(1)=0, \qquad \coch(2)=1,\qquad \coch(3)=1,\qquad \coch(4)=2.
\ee 
Thus,  $\coch(w_1)= 0 +1+1+2  = 4.$ For $w_2$,  we have
\be
\coch(1)=0, \qquad \coch(2)=1, \qquad \coch(3)=1,  \qquad  
\ee 
so $\coch (w_2) = 0+1+1 = 2$. Similarly, we find $\coch(w_3)=0$.  Hence, the cocharge of $w$ is $$\coch(w)=4+2+0=6.$$
\end{example}

\subsection{The Lie algebras \texorpdfstring{$\g$}{} and \texorpdfstring{$\widetilde \g$}{}. }
Let $\mathfrak g$  denote the Lie algebra   $\sL_{n+1}(\mathbb C)$ of traceless $(n+1)\times (n+1)$ 
 complex matrices. Denote the set  $\{1,\,\dots,n\}$  by $I$ and  the set  $\{0,1,\,\dots,n\}$ by $\widetilde I$.

Let $E_{i,\,j}$ be the $(n+1)\times (n+1)$ matrix unit with 1 at the $i,j$ position and $0$ elsewhere. The standard Chevalley generators of $\g$ are given by the matrices
 \be
 h_i= E_{i,\,i}-E_{i+1,\,i+1}, \qquad x^+_i = E_{i,\,i+1}, \qquad x^-_i= E_{i+1,\,i}, \qquad i\in I.
 \ee
 
Denote by $\n^+,\n^-, \h$ the Lie subalgebras of upper triangular, lower triangular and diagonal traceless matrices, respectively.

 Let $\{\omega_i\}_{i\in I}$ be the fundamental weights of $\g$ satisfying $\omega_i(h_j) = \delta_{i,j}$ for all $i,j\in I$. Set $\omega_0 =  \omega_{n+1} = 0$ to be zero weights. Let $\{\alpha_i\}_{i\in I }$ be the set of simple roots of $\mathfrak g$, where $\alpha_i = 2\omega_i - \omega_{i-1}-\omega_{i+1}$. 

 Let $P$ and $Q$ (respectively, $P^+$ and $Q^+$) be the weight and root lattices (respectively, the cones of dominant weights in the the weight and root lattices)
 formed by all linear combinations of $\omega_i$ and $\al_i$, $i\in I$, with integer (respectively, non-negative integer) coefficients.

 For $\la\in P^+$, define the integer
$$
|\lambda|=\sum_{i=1}^n i \lambda(h_i)\in \mathbb Z_{\geq 0}.
$$

For each $\la\in P^+$, define a partition of length at most $n$ by the rule
\begin{equation*} \label{partincl}
 \lambda \mapsto \Big( \lambda(h_1)+\lambda(h_2)+\cdots+\lambda(h_n),\, \lambda(h_2)+\cdots+\lambda(h_n),\ \dots \ ,\, \lambda(h_n)\Big) \in \mathscr P_{|\lambda|}.
\end{equation*}
Conversely, given a partition $\lambda$ of length at most $n$, define a dominant weight by
$$\la \mapsto \mbox{wt}(\lambda) =  \sum_{j=1}^{{\ell(\la)}} (\lambda_j-\lambda_{j+1})\,\omega_j
\in P^+.$$
Thus, we identify the set of dominant weights $P^+$ with the set of partitions  of length at most $n$. 

Let $\preceq$ be the partial order of $P$ defined  by $\lambda \preceq \mu$ if and only if $\mu -\lambda \in Q^+$. This partial order  coincides with  the partial order $\preceq$ previously defined on the set of partitions. 

The root  $\theta= \sum_{i\in I}\al_i$ is the maximal root of $\g$ with respect to $\preceq$. Denote by  $x_\theta^+= E_{1,\,n+1}$,  $x_\theta^-= E_{n+1,\,1}$ the corresponding root vectors associated to $\theta$.

\medskip

 Let $\mathbb C[t^{\pm 1}]$ be  the algebra of Laurent polynomials in one variable. The loop algebra over $\mathfrak g$ is the vector space $\widetilde{\mathfrak g} = \mathfrak g \otimes \mathbb C[t^{\pm1}]$ with the Lie bracket   
 $$[x\otimes t^r, y\otimes t^s] =  [x,y]\otimes t^{r+s}, \qquad x,y \in \mathfrak g,\  r,s \in \mathbb Z.
 $$ 

Let $\widetilde{\mathfrak n}^\pm = {\mathfrak n}^\pm \otimes \mathbb C[t^{\pm1 }], \ \widetilde{\mathfrak h} = \mathfrak h \otimes \mathbb C[t^{\pm1}]$.
For $i\in I$ and $r\in \mathbb Z$, we simplify notation by setting
$$
x_{i,r}^{\pm } = x_i^{\pm }\otimes t^r,\qquad h_{i,r}= h_i\otimes t^r.
$$

 We identify $\g$ with the  Lie subalgebra $\g\otimes 1 \subseteq \widetilde \g$. Under this identification, we further abbreviate the elements $x_{i,0}^\pm$, $h_{i,0} \in \widetilde \g$ to  $x_i^\pm$, $h_i$, respectively. Let $x_0^\pm = x_\theta^\mp \otimes t^{\pm 1}\in \widetilde \g$.
 The set $\{x_i^\pm, h_i\}_{i\in \widetilde I}$ generates $\widetilde \g$ as a Lie algebra.

Let $U(\g)$ and $U(\widetilde \g)$ be the universal enveloping algebras of $\g$ and $\widetilde \g$, respectively.

\subsection{Finite-dimensional \texorpdfstring{$\g$}{}-modules}

Let $\mbox{Rep}(\mathfrak g)$ be the abelian category of the finite-dimensional left $\mathfrak{g}$-modules. The category $\mbox{Rep}(\g)$ is semisimple and the isomorphism classes of irreducible objects are in bijection with $P^+$. 

For $\lambda \in P^+$, denote by $V(\lambda)$ the irreducible $\g$-module of highest weight $\lambda$. The module $V(\lambda)$ is the unique $\g$-module generated by a non-zero vector $v$ satisfying
\begin{equation} \label{highest weight}
    h_i\cdot  v = \lambda(h_i)v, \qquad  \n^+\cdot v =0, \qquad (x_i^-)^{\lambda(h_i)+1}\cdot v =0, \qquad   i \in I.
\end{equation}

For $V\in \mbox{Rep}(\mathfrak g)$ and $\lambda \in P^+$, let $[V: V(\lambda)]\in \mathbb Z_{\geq 0}$ be the multiplicity of $V(\lambda)$ in $V$.  For  $\mu \in \mathfrak h^*$, the subspace of $V$ of weight $\mu$ is 
$$V_\mu = \{v\in V: h\cdot v = \mu(h)v \,\,\,\,\mbox{for all}\,\,\,\, h\in \mathfrak h\}.$$

Let $\omega:\g \to \g$ be the antiautomorphism of Lie algebras given by the transposition of matrices. We have $\omega(x_i^{\pm }) = x_i^{\mp }, \, \omega(h_i) = h_i, \ i \in I.$ 
For $V\in \mbox{Rep}(\g)$, let $V^{\,\vee}\in \mbox{Rep}(\g)$ denote the dual module. As a vector space, the module $V^{\vee}$ is dual to $V$ and the left $\g$-action is given by  
\begin{equation*}
    (x\cdot f)(v) = f(\omega(x)\cdot v), \qquad  x\in \g, \ f\in V^\vee, \ v\in V. 
\end{equation*}
We have the isomorphism of $\g$-modules 
$$V(\la)^{\vee}\simeq V(\la),\qquad  \la \in P^+. $$

\subsection{Finite-dimensional \texorpdfstring{$\widetilde \g$}{}-modules.} Let $\mbox{Rep}(\widetilde \g)$ be the abelian category of the finite-dimensional left $\widetilde \g$-modules. The category  $\mbox{Rep}(\widetilde \g)$ is not semisimple.

For $a\in \mathbb C^\times$, the evaluation map $\ev_a: \widetilde \g \rightarrow \g$ is the Lie algebra homomorphism given by 
$$
x\otimes t^{r} \mapsto a^r x, \qquad x\in \g,\ r\in \mathbb Z.
$$ 
For $V\in \mbox{Rep}(\g)$, denote by $V_a \in \mbox{Rep}(\widetilde \g)$  the evaluation module obtained from $V$ by the pullback  along $\mbox{ev}_a$. For $\lambda \in P^+$, we simply denote by $V_a(\la)$ the  evaluation module $(V(\lambda))_a$ obtained from the irreducible $\g$-module $V(\lambda)$. 

It is well-known that any non-trivial finite-dimensional irreducible $\widetilde \g$-module is isomorphic to a tensor product of the form, see \cite{CP86},
\begin{equation*} \label{eval}
    V_{a_1}(\lambda_1)\otimes \cdots \otimes V_{a_k}(\lambda_k), \qquad \mbox{where}\quad \lambda_i \in P^+\setminus\{0\},\ a_i \in \mathbb C^\times, \ a_i\neq a_j \ \mbox{ for all }  \ i\neq j.
\end{equation*}

Moreover, any two such tensor products are isomorphic as $\widetilde \g$-modules if and only if they are obtained from each other by a permutation of tensor factors.

Let $\widetilde \omega: \widetilde \g \to \widetilde \g$ be the unique antiautomorphism of Lie algebras extending $\omega:\g\to \g$ and such that 
$$\widetilde\omega (x\otimes t^r) = \omega(x)\otimes t^{-r},\qquad r\in \mathbb Z.$$   
For $W\in \mbox{Rep}(\widetilde \g)$, let $W^{\,\vee}\in \mbox{Rep}(\widetilde\g)$ denote the dual module. As a vector space, the module $W^{\vee}$ is dual to $W$ and the left $\widetilde \g$-action is given by  
\begin{equation} \label{tildeomega}
    (x\cdot f)(w) = f(\widetilde \omega(x)\cdot w), \qquad x\in \widetilde \g, \ f\in W^\vee, \ w\in W.
\end{equation} 
We have the isomorphism of $\widetilde \g$-modules
\begin{equation} \label{veedual}
    V_a(\la)^{\vee}\simeq V_{a^{-1}}(\la), \qquad a\in \mathbb C^\times, \ \la\in P^+.
\end{equation}

For $\lambda \in P^+\setminus\{0\}, a\in \mathbb C^\times$, define the local Weyl module $W_a(\lambda)$ as the quotient of $U(\widetilde \g)$ by 
the left ideal generated by the elements
\begin{equation*}  h_{i,r}- a^r\lambda(h_i), \qquad  x_{i,r}^+, \qquad  (x_i^-)^{\lambda(h_i)+1}, \qquad  r\in \mathbb Z,\ i\in I.
\end{equation*}

 Recall that the socle of a module $M$ is the maximal semisimple submodule of $M$ and that the head of $M$ is the maximal semisimple quotient of $M$.

 The following theorem summarizes well-known facts about the local Weyl modules.
\begin{thm} We have the following properties of the local Weyl modules. \label{localweyl}
\begin{enumerate}
\item The local Weyl module $W_a(\lambda)$ is a non-zero cyclic finite-dimensional $\widetilde \g$-module. 
\item As a $\g$-module, the local Weyl module is isomorphic to the tensor product of fundamental $\g$-modules corresponding to the columns of the partition $\la$
   $$W_a(\lambda) \simeq V(\omega_1)^{\otimes\lambda(h_1)}\otimes \cdots \otimes V(\omega_n)^{\otimes\lambda(h_n)}.$$
   \item If $W$ is a $\widetilde \g$-module which has a cyclic non-zero vector $w$ satisfying the relations
\begin{equation*} \label{weylrel} h_{i,r}\cdot w =   a^r\lambda(h_i),  \qquad  x_{i,r}^+\cdot w = 0, \qquad  (x_i^-)^{\lambda(h_i)+1}\cdot w = 0, \qquad  r\in \mathbb Z,\ i\in I,\end{equation*}
then there exists a surjective map of $\widetilde \g$-modules $W_a(\la)\twoheadrightarrow W$. 
\item The head of the local Weyl module $W_a(\la)$ is isomorphic to $V_a(\la)$. If $|\la|\leq n$, then  the socle  of the local Weyl module $W_a(\la)$ is isomorphic to $V_a(\omega_d)$, where $d = |\la|$. 

\item Let $V$ be a $\widetilde \g$-module such that $V\simeq W_a(\la)$ as $\g$-modules.  Suppose that the head of $V$ contains $V_a(\la)$ and that the socle of $V$ is $V_a(\omega_d)$, where $d = |\la|\leq n$. Then, $ V\simeq W_a(\la)$ as $\widetilde \g$-modules.
   \end{enumerate}
\end{thm}
\begin{proof}
    Properties (i), (ii) can be found in \cite{CP}, \cite{CL}. Property (iii) is obvious. Property (iv) is found in \cite{KN}. We now prove (v).

    Since the head of $V$ contains $V_a(\la)$, there exists a short-exact sequence of $\widetilde\g$-modules
    $$0\to \Ker \pi \to V \xrightarrow{\pi} V_a(\la) \to 0.$$
    From the $\g$-modules isomorphism $V\simeq W_a(\la)$, we know that  $\dim V_\la = 1$. In particular, we must have $\pi(v_0) = v_0'$, where $v_0\in V_\la \subset V$ and $v_0'\in V_a(\la)$ are the corresponding highest weight vectors. Set $\widetilde V = U(\widetilde \g)\cdot v_0\subseteq V.$ We show that $v_0$ satisfies the relations in (iii). Recall that $v_0$ satisfy the relations in \eqref{highest weight}. 
    Clearly, we have
$$h_{i,r}v_0-a^r\la(h_i)v_0\in V_\la \cap \Ker \pi = \{0\}.$$
  Since  $x_{i,r}^+ = \frac{1}{2}[h_{i,r}, x_i^+]$ in $\widetilde \g$, the relation $x_{i,r}^+\cdot v_0  = 0$ follows from $x_i^+\cdot v_0 = 0$. Therefore, $\widetilde V$ is a quotient of $W_a(\la)$. Since the socle of $V$ is $V_a(\omega_d)$, then $V_a(\omega_d)\subseteq \widetilde  V$. Now, since $W_a(\la)$ has socle $V_a(\omega_d)$ and $[W_a(\la):V(\omega_d)]=1$, any non-trivial quotient of $W_a(\la)$ does not contain $V(\omega_d)$. Therefore, it follows that $\widetilde V \simeq W_a(\la)$ as $\widetilde \g$-modules. Now, since $\dim V = \dim W_a(\la)$, we conclude that $\widetilde V = V$.
\end{proof}

\medskip

Now, we define the general local Weyl module. First we prepare some notation.

For a partition $\la$ and a sequence of complex numbers $\bs a = (a_1,\cdots, a_{\la_1})$, we refer to $\bs a$ as the sequence of evaluation parameters of $\la$ and we think of the number $a_i$ as a label for the $i$-th column of the Young diagram of $\la$.

For $\la \vdash d_1, \ \mu \vdash d_2$ and $\bs a = (a_1,\dots, a_{\la_1}), \bs b = (b_1,\dots, b_{\mu_1}),$ we define the concatenation partition $\nu = \la \sqcup \mu$ and the concatenated sequence of evaluation parameters $\bs c = \bs a \sqcup \bs b$ as follows. Define $\nu  \vdash  (d_1+d_2)$ to be the partition whose columns are the union of the columns of $\la$ and $\mu$. Let each column of $\nu$ keep the same evaluation parameter it had in $\la$ or $\mu$, and let $\bs c$ be the sequence of evaluation parameters corresponding to the columns of $\nu$. 

For $\la \vdash d$ and $\bs b = (b_1,\dots, b_{\la_1})$, we may write $\la$ and $\bs b$ as a concatenation by collecting columns associated to the same evaluation parameters. More precisely, let $\{a_1,\dots, a_s\}$ be the set of distinct evaluation parameters among all $b_i$ and let $\la^{(i)}$ be the partition whose columns are the columns of $\la$ associated with the evaluation parameter $a_i$. Then, we have $\la = \la^{(1)}\sqcup \cdots \sqcup \la^{(s)}$ and $\bs b = \bs a_1\sqcup \cdots \sqcup \bs a_s$, where $\bs a_i = (a_i,\dots, a_i)$. We refer to the partitions $\la^{(i)}$ and sequences of evaluation parameters $\bs a_i$ as the splitting of $\la$ and $\bs b$. The splitting of $\la$ and $\bs b$ is unique up to simultaneous permutation of $\la^{(i)}$ and $a_i$.

For $\la \vdash d$ and $\bs b$, we define the local Weyl module $W_{\bs b}(\la)$ as follows.  Let $\la = \la^{(1)}\sqcup \cdots \sqcup \la^{(s)}$ and $ \bs b = \bs a_1\sqcup \cdots \sqcup \bs a_s$ be the  splitting of $\la $ and $\bs b, $ with $\bs a_i = (a_i,\dots,a_i)$. Then, the tensor product of local Weyl modules  \begin{equation}\label{weylmodule}
    W_{\bs b}(\la)=W_{a_1}(\lambda^{(1)})\otimes \cdots \otimes  W_{a_s}(\lambda^{(s)})
\end{equation}  is called 
the local Weyl module of highest $\ell$-weight  \begin{equation*} 
    \bigg(\sum_{j=1}^s\frac{\la^{(j)}(h_1)}{x-a_j},\dots, \sum_{j=1}^s\frac{\la^{(j)}(h_n)}{x-a_j} \bigg),
\end{equation*}
where $x$ is a  formal variable.

Alternatively, the local Weyl module $W_{\bs b}(\la)$ is isomorphic to the quotient of  $U(\widetilde \g)$ by the left ideal generated by the elements
\begin{equation*}  h_{i,r} -\sum_{j=1}^s \lambda^{(j)}(h_i)a_j^r, \qquad  (x_i^-)^{1+\sum_{j=1}^s\lambda^{(j)}(h_i)}, \qquad  {x_{i,r}^+}, \qquad  r\in \mathbb Z, \ i\in I.\end{equation*}
Moreover, the tensor product of evaluation modules $V_{a_1}(\lambda^{(1)})\otimes \cdots \otimes  V_{a_s}(\lambda^{(s)})$ is the corresponding unique irreducible quotient of $ W_{\bs b}(\la)$.

\subsection{Symmetric group and extended affine symmetric group}
Let $S_d$ be the symmetric group on $d$ elements. The group $S_d$ is generated by  the simple transpositions $\sigma_i = (i,i+1), 1\leq i <d,$ satisfying the defining relations
\begin{equation} \label{symrel}
    \sigma_i^2=1,  \qquad \sigma_i\sigma_{i+1}\sigma_i = \sigma_{i+1}\sigma_i\sigma_{i+1}, \qquad  \sigma_i\sigma_j = \sigma_j\sigma_i, \quad  |i-j|> 1.
\end{equation}

The group $S_d$ has a right action on $\mathbb Z^d$ of permuting the coordinates.  The extended affine symmetric group is the semi-direct product $\widetilde S_d = S_d\ltimes \mathbb Z^d $. Let $t_j^{\pm 1} = (1, \pm \bs{e}_j)\in \widetilde{S}_d$, where $\bs e_j$ is the $j$-th standard generator of $\mathbb Z^d$. Then, the group $\widetilde S_d$ is 
 generated by the elements $\sigma_i,\ 1\leq i <d,\,\, t_j^{\pm 1},\  1\leq i \leq d,$ satisfying the defining relations (\ref{symrel}) together with
\begin{align*} \label{extsymrel}
		t_it_j=t_jt_i, \qquad t_it_i^{-1} = t_i^{-1}t_i = 1, \qquad t_j\sigma_i = \sigma_it_j,\ \ j\not\in \{i,\,i+1\}, \qquad \sigma_it_i = t_{i+1}\sigma_i.
	\end{align*}

Let $\C S_d, \C\Z^d,\C\widetilde S_d$ be the corresponding group algebras. We identify $\mathbb C \mathbb Z^d$ with the algebra of multivariate Laurent polynomials $ \mathbb C[t_1^{\pm 1}, \dots, t_d^{\pm 1}]$. Then,
\be
\mathbb C \widetilde S_d= \mathbb C S_d \otimes \mathbb C[t_1^{\pm 1},\dots,t_d^{\pm 1}]
\ee
with the product given by
\begin{equation} \label{semidirect}
    (\sigma \otimes f)(\tau \otimes g) = \sigma\tau \otimes (f\cdot \tau)g, \qquad f,g \in \mathbb C[t_1^{\pm 1}, \dots, t_d^{\pm 1}],\ \sigma, \tau \in \mathbb CS_d. 
\end{equation}
We often omit the symbol $\otimes $ and simply write $\sigma f \in  \mathbb C \widetilde S_d$ instead of $\sigma \otimes f$.

\subsection{Finite-dimensional \texorpdfstring{$S_d$}{}-modules} Let $\mbox{Rep}(S_d)$ be the abelian category of the  finite-dimensional right $S_d$-modules.  The category $\mbox{Rep}(S_d)$ is semisimple and the isomorphism classes of irreducible $S_d$-modules are in bijection with $\mathscr P_d$. 

For $\lambda\vdash d$, denote by $L(\la)$ the corresponding irreducible $S_d$-module.  The module $L(\lambda)$ is the cyclic $S_d$-module generated by the Young symmetrizer of $\lambda$, see \cite{FH} for an explicit construction. 

The $S_d$-modules $L(d)$ and $L(1^{\{d\}})$ are the $1$-dimensional trivial and sign representations, respectively. For $\la\vdash d$, we have the isomorphism of $S_d$-modules
\begin{equation} \label{transpose}
    L(\lambda)\otimes L(1^{ \{ d \}}) \simeq L(\lambda^t).
\end{equation}

 Let $\tau: S_d\to S_d$ be the group antiautomorphism  given by the inversion $\tau(\sigma) = \sigma^{-1}, \ \sigma \in S_d.$ For $L\in \mbox{Rep}(S_d)$, let $L^{\,\vee}\in \mbox{Rep}(S_d)$ denote the dual module. As a vector space, the module $L^{\vee}$ is dual to $L$ and the right $S_d$-action is given by  
 \begin{equation*}
     (f\cdot \sigma)(u) = f(u\cdot \tau(\sigma)), \qquad \sigma \in S_d, \ f\in L^\vee, \ u\in L.
 \end{equation*}
We have the isomorphism of $S_d$-modules
$$L(\la)^{\vee} \simeq L(\la), \qquad \la \vdash d. $$

Let $V\in \mbox{Rep}(S_d)$.  Denote by $[V:L(\lambda)]\in \mathbb Z_{\geq 0 }$ the multiplicity of $L(\lambda)$ in $V$. We view the character of $V$ as the $\mathbb Z_{\geq 0}$-valued  map 
$$
\mbox{ch}\,V:\ \mathscr P_d \rightarrow \mathbb Z_{\geq 0}, \qquad \lambda\mapsto [V:L(\lambda)].
$$
Let $\chi_\lambda=\mbox{ch}\, L(\lambda)$. Then, 
\begin{equation*}
    \label{ch}  \mbox{ch}\,V = \sum_{\lambda \, \vdash\, d} [V:L(\lambda)] \chi_\lambda.
\end{equation*} 
Clearly, $V,W\in \mbox{Rep}(S_d)$ are isomorphic as $S_d$-modules if and only if $\mbox{ch}\, V = \mbox{ch}\, W$.

\subsection{Finite-dimensional representations of \texorpdfstring{$\widetilde S_d$}{}}
Let $\mbox{Rep}(\widetilde S_d)$ denote the abelian category of the finite-dimensional right $\widetilde S_d$-modules. The category $\mbox{Rep}(\widetilde S_d)$ is not semi-simple. 

For $a\in \mathbb C^\times$, the evaluation map $\mbox{ev}_a: \mathbb C  \widetilde S_d \rightarrow \mathbb C S_d$ is the algebra homomorphism satisfying
$$\sigma_i\mapsto \sigma_i, \quad  1\leq i<d, \qquad t_j^{\pm 1} \mapsto a^{\pm 1}, \quad 1\leq j \leq d.$$
Given  $M\in \mbox{Rep}(S_d)$, denote by $M_a \in \mbox{Rep}(\widetilde S_d)$ the evaluation module obtained by the pullback along $\mbox{ev}_a$. For $\lambda \vdash d$, we simply denote by $L_a(\lambda)$  the evaluation module  $(L(\lambda))_a$.  

It is known that any non-trivial finite-dimensional irreducible  $\widetilde S_d$-module is isomorphic to an affine Zelevinsky tensor product of the form, see \cite{Y},
$$L_{a_1}(\lambda_1)\widetilde \boxtimes \cdots \widetilde \boxtimes L_{a_k}(\lambda_k), \qquad \mbox{where}\quad \lambda_i \vdash d_i, \,\,\, d = d_1+\dots+d_k,\,\,\, a_i\in \mathbb C^\times,  \ a_i\neq a_j \ \mbox{ for all } \ i\neq j.$$
 Moreover, any two such tensor products are isomorphic as $\widetilde S_d$-modules if and only if they are obtained from each other by a permutation of tensor factors.

Let $\widetilde \tau: \widetilde S_d\to \widetilde S_d$ be the antiautomorphism of groups extending $\tau: S_d\to S_d$ and such that
$$\widetilde \tau(t_j^{\pm 1}) = t_j^{\mp 1}, \qquad 1\leq j \leq d. $$
For $M\in \mbox{Rep}(\widetilde S_d)$, let $M^{\,\vee}\in \mbox{Rep}(\widetilde S_d)$ denote the dual module. As a vector space, the module $M^{\vee}$ is dual to $M$ and the right $\widetilde S_d$-action is given by  
\begin{equation} \label{tildetau}
    (f\cdot s)(m)  = f(m\cdot\widetilde{\tau}(s)),\qquad s\in \widetilde S_d, \ f\in M^\vee, \ m\in M.
\end{equation}
We have the isomorphism of $\widetilde S_d$-modules
$$ L_a(\la)^{\vee} \simeq L_{a^{-1}}(\la). $$

\subsection{Zelevinsky tensor products}

Let $d = d_1+\dots+d_\ell$, $d_i\in \mathbb Z_{>0}$, be a composition of the integer $d$. We have the following  natural inclusion of algebras
\begin{align}
\begin{split} \label{affzel}
     \widetilde \iota :\  & \C \widetilde S_{d_1} \otimes \cdots \otimes \C \widetilde S_{d_\ell} \hookrightarrow \C \widetilde S_d, \\   
&1^{(j-1)}\otimes \sigma_i \otimes 1^{(\ell-j)} \,\,\mapsto\,\, \sigma_{i+d_1+\,\cdots\,+d_{j-1}}, \quad 1 \leq j \leq \ell,\,\, 1\leq i < d_j,\\
&1^{(j-1)} \otimes t_i \otimes 1^{(\ell-j)}\,\, \mapsto\,\, t_{i+d_1+\,\cdots\,+d_{j-1}}, \quad 1 \leq j \leq \ell, \,\, 1 \leq i \leq d_j.
\end{split}
\end{align}

For $ M_i\in \mbox{Rep}(\widetilde S_{d_i}), \ 1\leq i \leq \ell$, the affine Zelevinsky tensor product 
is the induced right $\widetilde S_d$-module
\begin{equation*} 
    M_1\widetilde \boxtimes \cdots \widetilde \boxtimes  M_\ell =  \mbox{Ind}^{\mathbb C \widetilde S_d}_{\mathbb C \widetilde S_{d_1}\otimes \, \cdots \, \otimes \, \mathbb C \widetilde S_{d_\ell}}( M_1\otimes \cdots \otimes  M_\ell).
\end{equation*} 

Let $\iota: \mathbb C S_{d_1}\otimes \cdots \otimes \mathbb C S_{d_\ell} \hookrightarrow \mathbb C S_d$ be the restriction of $\widetilde \iota$ to the subalgebra $\C S_{d_1}\otimes \cdots \otimes \C S_{d_\ell}$. For  $L_i\in \mbox{Rep}(S_{d_i}), 1\leq i \leq \ell$,  the Zelevinsky tensor product 
is the right $S_d$-module
$$ L_1\boxtimes \cdots \boxtimes L_\ell = \mbox{Ind}^{\mathbb C S_d}_{\mathbb C S_{d_1}\otimes \, \cdots \, \otimes \, \mathbb C S_{d_\ell}}(L_1\otimes \cdots \otimes L_\ell).$$

We note that the Zelevinsky tensor products $\widetilde \boxtimes, \ \boxtimes$ are commutative and associative up to isomorphism. Moreover, if $ M_i\in \mbox{Rep}(\widetilde S_d)$ are such that $ M_i \simeq L_i$ as $S_d$-modules, then there is an isomorphism of $S_d$-modules
\begin{equation} \label{forgetzele}
     M_1\widetilde \boxtimes \cdots \widetilde \boxtimes  M_\ell \simeq  L_1\boxtimes \cdots \boxtimes L_\ell.
\end{equation}

\subsection{Deformation lemmas}
In Section 5.1, we will use the following standard lemmas which deal with continuous deformations of algebras and $S_d$-modules.

 Let $R=\mathbb C[t_1,\dots,t_d]$ and let $f_1(\epsilon), \dots, f_s(\epsilon)\in R$ be polynomials which continuously depend on a parameter $\epsilon\in\R$. Denote by $I_\epsilon$ the ideal generated by  $f_1(\epsilon), \dots, f_s(\epsilon)$.
\begin{lem}\label{deformed ideal lemma}
Let $\epsilon_0>0$. If $\dim R/I_\epsilon = k $ holds  for all $0 < \epsilon<\epsilon_0$, then $\dim R/I_0\geq k$. \qed
\end{lem}

Let $\{V_{\bs a}\}_{\bs a \in \mathbb (\mathbb C)^p}$ be a family of finite-dimensional $S_d$-modules depending on a parameter $\bs a \in \mathbb (\mathbb C)^p$. Suppose that $V_{\bs a} \simeq  V_0$ as a vector space for all $\bs a$. Let $\{v_1,\dots,v_k\}$ be a basis for $V_0$. 

\begin{lem}\label{deformed module lemma}
If the action of $S_d$ in $V_{\bs a}$ written in the basis $\{v_1,\dots,v_k\}$ depends on $\bs a$ continuously, then $V_{\bs a} \simeq V_0$ as $S_d$-modules for all $\bs{a} \in \mathbb (\mathbb C)^p$. \qed
\end{lem}

\section{Classical and Affine Schur-Weyl duality}  \label{SW}

\subsection{The classical Schur-Weyl duality} Let $\mbox{Rep}_d(\g)$ be the full subcategory of $\mbox{Rep}(\g)$ given by the $\g$-modules $V$ satisfying the condition
$$[V:V(\lambda)]\neq 0 \implies |\lambda| =  d. $$

Let $\mathbb V = \mathbb C^{n+1}\simeq V(\omega_1)$ denote the standard $\g$-module. The group $S_d$ acts on the left of the $d$-th tensor power $\mathbb V^{\otimes d}$ by permuting tensor entries. This action commutes with the  usual left $\g$-action on $\mathbb V^{\otimes d}$ given by the coproduct $\Delta(x) = \sum_{j=1}^d \Delta_j(x)$, where $\Delta_j(x) = 1^{\otimes (j-1)}\otimes x \otimes 1^{\otimes (d-j)}, \ x\in \g$.

For $L\in \mbox{Rep}(S_d)$,  let
\begin{equation*} \label{functormodule}
	\mathcal F_d(L) = L \, \bigotimes_{\mathbb CS_d} \mathbb V^{\otimes d}.
\end{equation*} 
Regard $\mathcal F_d(L)$ as a $\g$-module by the inherited  action  $$x\cdot (u\otimes w) = u \otimes (x\cdot w),\qquad  x\in \mathfrak g,\ u\in L,\ w \in \mathbb V^{\otimes d}.$$ If $d\leq n$,  we have  an isomorphism of $\g$-modules, see \cite{FH},
\begin{equation} \label{gisos}
    \mathcal F_d( L(\lambda)) \simeq V(\mbox{wt}(\lambda)), \qquad \la \vdash d.
\end{equation}

Let  $\mathcal F_d: \mbox{Rep}(S_d) \rightarrow \mbox{Rep}_d(\g)$ be the functor of abelian categories sending an $S_d$-module $L$ to $\mathcal F_d(L)$ and a homomorphism of $S_d$-modules $f: L\to K$  to $f\otimes 1: \mathcal F_d(L) \to \mathcal 
F_d(K)$, where $1:\mathbb V^{\otimes d}\to \mathbb V^{\otimes d}$ is the identity map.

The functor $\mathcal F_d$ is exact and compatible with the Zelevinsky tensor product
\begin{equation} \label{functorzele}
    \mathcal F_d(L_1\boxtimes L_2) \simeq \mathcal F_{d_1}(L_1)\otimes \mathcal F_{d_2}(L_2), \qquad L_i\in \mbox{Rep}(S_{d_i}), \ d = d_1+d_2.
\end{equation}

\medskip

The next theorem is the categorical formulation of the well-known classical Schur-Weyl duality.

\begin{thm} \label{classchur} \normalfont
	If  $d \leq n$, the functor $\mathcal F_d: \mbox{Rep}(S_d) \rightarrow \mbox{Rep}_d(\g)$ defines an equivalence of categories. \qed
\end{thm}

\subsection{The Affine Schur-Weyl duality}   
Let $\mbox{Rep}_d(\widetilde \g)$ be the full subcategory of $\mbox{Rep}(\widetilde \g)$ consisting of the $\widetilde \g$-modules $V$ such that $V\in \mbox{Rep}_d(\g)$  as a $\g$-module.

The vector space $\mathbb V[t^{\pm 1}] =\mathbb C^{n+1 }\otimes \mathbb C[t^{\pm 1}]$ is a left $\mathbb C[t^{\pm 1}]$-module with the action 
$$
f\cdot (v\otimes g) = v\otimes (fg), \qquad v\in \mathbb C^{n+1}, \ f,g\in \mathbb C[t^{\pm 1}],
$$
as well as a left $\widetilde \g$-module with the action
$$(x\otimes f)\cdot (v\otimes g) =  (x\cdot v)\otimes (fg), \qquad x\in \g.$$

Define the right action of the extended affine symmetric group $\widetilde S_d$ on
the $d$-th tensor power $(\mathbb V[t^{\pm 1}])^{\otimes d}$ by letting $S_d$ act by permuting tensor factors and  $t_j^{\pm 1}$ act by 
$$t_j^{\pm 1}\cdot (w_1\otimes \dots \otimes w_d) = w_1\otimes \dots \otimes (t^{\pm 1}\cdot w_j) \otimes \dots \otimes w_d, \qquad w_i \in \mathbb V[t^{\pm 1}], \ 1\leq i \leq d. $$

For $ M\in \mbox{Rep}(\widetilde S_d)$, let 
\begin{equation*}
\widetilde{ \mathcal F}_d(M)  = M \, \bigotimes_{\mathbb C\widetilde S_d} (\mathbb V[t^{\pm 1}])^{\otimes d}.
\end{equation*} 
Regard $\widetilde{\mathcal F}_d(M)$ as a left $\widetilde \g$-module by the inherited action of $\widetilde \g$  on $(\mathbb V[t^{\pm 1}])^{\otimes d}$.

We have an isomorphism of $\g$-modules ${\mathcal F}_d(M) \to \widetilde{\mathcal F}_d(M)$ given by
$$m\otimes (v_1\otimes \cdots \otimes v_d) \mapsto m\otimes ((v_1\otimes 1)\otimes \cdots \otimes (v_d\otimes 1)), \qquad  m\in M, \ v_i\in \mathbb V.$$  It follows that $\widetilde{\mathcal F}_d(M)$ is isomorphic to the the $\widetilde \g$-module obtained from ${\mathcal F}_d(M)$ by extending the $\g$-action to a $\widetilde\g$-action via
\begin{equation} \label{x_0action}
    x_0^{\pm}\cdot(m\otimes w) = \sum_{j=1}^d m\cdot t_j^{\pm 1}\otimes \Delta_j(x_\theta^{\mp})(w), \qquad m\in M, \ w\in \mathbb V^{\otimes d}.
\end{equation}

Let  $\widetilde{\mathcal F}_d: \mbox{Rep}(\widetilde S_d) \rightarrow \mbox{Rep}_d(\widetilde \g)$ be the functor of abelian categories sending an $\widetilde S_d$-module $M$ to $\widetilde{\mathcal F}_d(M)$ and a homomorphism of $\widetilde S_d$-modules $f:\ M\to N$ to $f\otimes 1: \widetilde{\mathcal F}_d(M) \to \widetilde{\mathcal 
F}_d(N)$.

\medskip

The next theorem is known as the affine Schur-Weyl duality. 

\begin{thm}[\cite{CP96}, \cite{Y}] \label{affineschur} \normalfont
	If  $d \leq n$, the functor $\widetilde{\mathcal F}_d: \mbox{Rep}(\widetilde S_d) \rightarrow \mbox{Rep}_d(\widetilde \g)$ defines an equivalence of categories. \qed
\end{thm}

The functor $\widetilde{\mathcal F}_d$ is exact and compatible with both the affine Zelevinsky tensor product and the evaluation homomorphisms, see \cite{Y},
\begin{align}
& \widetilde{ \mathcal F}_d(L_a) \simeq (\mathcal F_d(L))_a, \qquad L\in \mbox{Rep}(S_d), \ a\in \mathbb C^\times,
\\
    &\widetilde{\mathcal F}_d(M_1\widetilde \boxtimes M_2) \simeq \widetilde{\mathcal F}_{d_1}(M_1)\otimes \widetilde{\mathcal F}_{d_2}(M_2), \qquad M_i\in \mbox{Rep}(\widetilde S_{d_i}), \ d = d_1+d_2. \label{functortensor}
\end{align}

The following proposition shows that $\widetilde{\mathcal F}_d$ is also compatible with taking dual modules through the antiautomorphisms  $\widetilde \tau : \widetilde S_d \to \widetilde S_d, \ \widetilde \omega:\widetilde \g \to \widetilde \g$ defined in \eqref{tildetau}, \eqref{tildeomega}, respectively.

\begin{prop}  For any $M\in \on{Rep}(\widetilde S_d)$, there is an isomorphism of $\widetilde \g$-modules
  $\widetilde{\mathcal F}_d(M^{\vee}) \simeq \widetilde{\mathcal F}_d(M)^{\vee}. $
\end{prop}

\begin{proof}
    Let $\{\bs v_1,\dots, \bs v_{d(n+1)}\}$ be a basis for $\mathbb V^{\otimes d}$ and let $\{\bs v_1^*,\dots, \bs v_{d(n+1)}^*\}$ be the corresponding dual basis for $\big(\mathbb V^{\otimes d}\big)^*=\big(\mathbb V^{\otimes d}\big)^\vee$.
Let $\varphi:  M^{\vee}\otimes_{\mathbb C S_d} \mathbb V^{\otimes d}\to \big(M\otimes_{\mathbb C S_d} \mathbb V^{ \otimes d} \big)^{\vee} $ be the isomorphism of vector spaces given by
$$\varphi(f\otimes \bs v_i)(m\otimes w) = f(m)\bs v_i^*(w), \qquad f \in M^{\vee},\ m\in M,\ w\in \mathbb V^{\otimes d}.$$
It can be verified that $\varphi(f\cdot \sigma \otimes \bs v_i) = \varphi(f\otimes\sigma \cdot \bs v_i)$  and that $\varphi(f\otimes \bs v_i)(m\cdot \sigma  \otimes w) = \varphi(f\otimes \bs v_i)(m\otimes \sigma \cdot w)$ for any $\sigma \in S_d$, so the map is well-defined. 

We are left to show that $\varphi$ commutes with the action of $\widetilde \g$. We check only for $x_0^+ = x_\theta^-\otimes t$, as the cases for $x\in \g$ and $x_0^-$ are analogous. Indeed, using \eqref{x_0action}, for all $m\in M, w\in \mathbb V^{\otimes d}$, we have
\begin{align*}
    & (x_0^+\cdot \varphi(f\otimes \bs v_i))(m\otimes w) = \varphi(f\otimes \bs v_i)( x_0^-\cdot (m\otimes w)) = \sum_{j=1}^d f(m\cdot t_j^{-1})\otimes {\bs v_i}^*(\Delta_j(x_\theta^+)(w)) \\
    & \hspace*{1.5cm} = \sum_{j=1}^d \varphi(f\cdot  t_j\otimes \Delta_j(x_\theta^-)\cdot {\bs v_i}^*)(m\otimes w) 
     =  \varphi(x_0^+\cdot (f\otimes \bs v_i))(m\otimes w).
\end{align*}
\end{proof}

\section{Garsia-Procesi-Tanisaki modules and their deformations} \label{TGP sec}

\subsection{Modified Kostka-Foulkes Polynomials}

Let $\lambda,\mu \vdash d$. For a semistandard Young tableau $T\in \mbox{SSYT}(\lambda, \mu)$,  the reading word  $\mbox{rw}(T)$ of $T$ is the word obtained by concatenating the fillings of the rows of $T$  from left to right and from bottom to top.

For instance, the word $w = 422311123$ in Example \ref{cocharge example} is the reading word of the following tableau
\vspace{0.1cm}
\begin{center}
T = \tiny \ytableaushort{11123,223,4} \, .
\end{center}
\vspace{0.25cm}

The cocharge of $T\in \mbox{SSYT}(\lambda, \mu)$ is defined as the cocharge of its reading word $$\coch(T)=\coch(\on{rw}(T)),$$
see Section \ref{sec cocharge}.

For $\lambda, \mu \vdash d$,  the modified Kostka-Foulkes polynomial $\widetilde K_{\lambda,\,\mu}(q)$ is the polynomial in a formal variable $q$ with non-negative integer coefficients given by
\be
\widetilde K_{\lambda,\,\mu}(q)=\sum_{T\in \on{\,SSYT}(\lambda,\,\mu)} q^{\on{coch}(T)}.
\ee 
The following properties of the modified Kostka-Foulkes polynomials are  straightforward.

\begin{enumerate}
    \item $\widetilde K_{\lambda,\,\mu}(q)\neq 0$ if and only if $\mu\trianglelefteq \lambda$.

    \item $\deg \widetilde K_{\lambda,\, \mu}(q)\leq n(\mu)$, where
$$n(\mu) = \sum_{j=1}^{\ell(\mu)} \mu_j(j-1). $$

\item $\widetilde K_{\lambda,\, \mu} (q)= 1$ if and only if $\lambda = (d)$.

\item $\widetilde K_{\lambda,\, \mu}(q) = q^{n(\mu)}$ if and only if $\lambda = \mu$. 

\item  The value of the  modified Kostka-Foulkes polynomial at $q=1$ is the Kostka number $\widetilde K_{\lambda,\, \mu}(1)= K_{\lambda, \,\mu}$.

\end{enumerate}

The usual Kostka-Foulkes polynomial  $K_{\lambda,\,\mu}(q)$ satisfies $K_{\lambda,\,\mu}(q)=q^{n(\mu)}\widetilde K_{\lambda,\,\mu}(q^{-1})$. In this article we only use the modified Kostka-Foulkes polynomials.


\begin{example} \label{kostka,d=4}
  Let $\mu = (1,1,1,1)$. We compute all the non-zero modified Kostka-Foulkes polynomials $K_{\lambda,\,\mu}(q)$. We first list all the   semi-standard Young tableaux of content $\mu$ as follows
\\
\begin{center} 
$T_1 = \tiny\ytableaushort{1,2,3,4}\,, \qquad T_2 = \ytableaushort{1234}\,, \qquad T_3 = \ytableaushort{12, 34}\,, \qquad T_{4} = \ytableaushort{13,24}\,,$
\end{center}
\vspace{0.3cm}
\begin{center} \hspace*{-1cm}
$\qquad T_5 =  \tiny\ytableaushort{134,2}\,, \qquad T_6 = \ytableaushort{124,3}\,, \qquad T_7 = \ytableaushort{123,4}\,, $
\end{center}
\vspace{0.3cm}
\begin{center} \hspace*{0cm}
$T_8 = \tiny\ytableaushort{12,3,4}\,, \qquad \quad\ T_9 = \ytableaushort{13,2,4}\,, \qquad\quad\ \   T_{10} = \ytableaushort{14,2,3}\,.\quad\  $
\end{center}

\vspace{0.3cm}
The cocharges of the above tableaux are
\begin{align*}
    &\coch(T_1) = 6, \qquad \coch(T_2) = 0, \qquad \coch(T_3) = 2, \qquad \coch(T_4) = 4, \qquad \coch(T_5) =  3, \\
    & \coch(T_6) = 2, \qquad \coch(T_7) = 1, \qquad \coch(T_8) = 3, \qquad \coch(T_9) = 4, \qquad \coch(T_{10}) = 5. 
\end{align*}
Thus, we have the following modified Kostka-Foulkes polynomials
$$\widetilde K_{\mu,\,\mu} = q^6,\quad \widetilde K_{(4),\,\mu}  = 1, \quad \widetilde K_{(2,2),\,\mu} = q^2+q^4, \quad \widetilde K_{(3,1),\,\mu} = q^1+q^2+q^3,\quad \widetilde K_{(2,1,1),\,\mu} = q^3+q^4+q^5. $$
\end{example}

\subsection{Tanisaki-Garsia-Procesi modules} \label{TGP section} Let $ t_1,\dots,t_d$ be  $d$ commuting variables, $d\in \mathbb Z_{>0}$. For  $J\subseteq \{1,\dots, d\}$, let 
$t_J = \{t_j: j\in J\}$. The $r$-th elementary symmetric polynomial in the variables $t_J$ is given by the sum 
$$e_r(t_J) = \sum_{i_1<i_2<\, \cdots \,<i_r}  t_{i_1}\dots t_{i_r}, $$
where the summation is over indexes $i_j \in J$.

Let $\lambda \vdash d$. For $0\leq n <\ell(\lambda)$, let $m_\lambda(n) = \sum_{i=n+1}^{\ell(\lambda)} \lambda_i$ be the number of boxes in the corresponding Young diagram of $\lambda$ in the rows numbered by $n+1,\dots, \ell(\lambda)$.  Note that $m_\lambda(0)= d$.

\bigskip

\begin{figure}[h!]
   \begin{tikzpicture} 
   \filldraw[gray!20] (0,0)--(0,2.4)--(0,2.4)--(3,2.4)--(3,2.4)--(3,1.5)--(3,1.5)--(2,1.5)--(2,1.5)--(2,0);
   \node at (-0.7,2.2) {$\la=$}; 
    \draw (0,0)--(0,4);
    \draw (0,0) -- (2,0);
    \draw (0,4) -- (4,4);
    \draw (4,4) -- (4,2.8);
    \draw (3,2.8) -- (4,2.8);
    \draw (3,2.8) -- (3,1.5);
    \draw (3,1.5) -- (2,1.5);

    \draw (2,1.5) -- (2,0);
    \draw[dashed] (0, 2.4) -- (5,2.4);
     \draw[dashed] (4, 4) -- (5,4);
    \node[anchor=west] at (0.4, 1.2) {$m_\la(n)$};
    \node[anchor=west] at (4.5, 3.2) {$n$};
    \draw[<->] (4.5,2.4) -- (4.5,4);
   
\end{tikzpicture} 
\caption{The partition $\la$ and $m_\la(n)$. }
\end{figure}

\bigskip

 For $0\leq n <\ell(\lambda)$, denote by $\mathbb J^{(d-n)} = \{J\subseteq \{1,\dots,d\}: \#J = d-n\}$ the set of all possible choices of $d-n$ elements in $\{1,\dots, d\}$. Let
$$\mathscr C_{\lambda, \,n}  = \{ e_r(t_J): J\in \mathbb J^{(d-n)}, \ \ d-n-m_\lambda(n)<r\leq d-n \}.$$
We define the set of Tanisaki polynomials $\mathscr C_\lambda$ by
$$\mathscr C_\lambda = \bigcup_{n =0 }^{\ell(\lambda)-1} \mathscr C_{\lambda, \,n}.$$

Let $\langle \mathscr C_\lambda \rangle\subseteq \mathbb C[t_1,\dots,t_d]$ be the ideal generated by $\mathscr C_\lambda$. The ideal  $\langle \mathscr C_\lambda\rangle$ is symmetric, i.e., it is invariant under the natural right action of $S_d$ on $\mathbb C[t_1,\dots, t_d]$. We refer to  the right $S_d$-module
$$R(\lambda)  = \mathbb C[t_1,\dots, t_d]/ \langle \mathscr C_\lambda \rangle$$
as the Tanisaki-Garsia-Procesi (TGP)  module $R(\la)$.\footnote{Our $R(\la)$ coincides with $R_{\la^t}$ in  \cite{GP}.}

If $\lambda, \mu\vdash d$ satisfy $\lambda \trianglelefteq \mu$, then $\mathscr C_{\mu}\subseteq \mathscr C_\lambda$. In particular, there exists a surjective homomorphism of $S_d$-modules $R({\mu})\twoheadrightarrow R(\lambda)$. 
In the case of  $\mu = (d),$ the ideal $\langle \mathscr C_{(d)} \rangle $ is the ideal of all symmetric polynomials and   $R{(d)}$ is also known as the coinvariant algebra. Any  TGP module $R({\lambda})$ with $\lambda\vdash d$ is a quotient of $R{(d)}$.  In particular,  $R(\lambda)$ is finite-dimensional with $\dim R(\lambda)\leq \dim R(d)=\ d!.$

The algebra $\mathbb C[t_1,\dots,t_d]$ has a  natural $\mathbb Z_{\geq 0}$-grading given by the homogeneous degree $\deg t_i=1$ for all $i$. This grading is $S_d$-invariant and the ideal $\langle \mathscr C_\la \rangle$ is also graded. Therefore,
the TGP module $R(\lambda)$ is a graded vector space 
$$
R(\lambda)=\bigoplus_{i=0}^\infty R(\lambda)[i],
$$
where each graded component $R(\lambda)[i]$ is an $S_d$-module. 
Define the graded character of $R(\lambda)$ in the formal variable $q$ by
$$
\on{gch} R(\lambda) = \sum_{\mu \, \vdash d}\bigg(\sum_{i=0}^\infty [R(\lambda)[i]: L(\mu)]q^i\bigg)\,\chi_{\mu}.
$$

The algebra $\mathbb C S_d \otimes \mathbb C[t_1,\dots,t_d]$ acts on the right of $R(\la)$ in the obvious way. Note that the grading of $R(\la)$ is compatible with the action of $\mathbb C S_d \otimes \mathbb C[t_1,\dots, t_d]$. One of the main results of \cite{GP} is that the graded character of $R(\lambda)$ is given in terms of the modified Kostka-Foulkes polynomials.

\begin{thm}[\cite{GP}]\label{GP thm} We have
$$\on{gch}\, R(\lambda) = \sum_{\mu \,\vdash d }\widetilde K_{\mu, \,\lambda^t}(q)\chi_\mu. 
$$
In particular, the graded component of maximal degree is $R(\la)[n(\la^t)] = L(\la^t)$.

\qed
\end{thm}

\begin{example}
Let $\lambda = (4)$. Then, $\la^t = (1,1,1,1)$ and from Example \ref{kostka,d=4} we obtain the graded character of the coinvariant algebra
\begin{align*}
\on{gch}\big(\C[t_1,t_2,t_3, t_4]/ \C[t_1,t_2,t_3, t_4]^{S_4}\big)=\on{gch} R(4) \hspace{8cm}\\ =\chi_{(4)}+(q+q^2+q^3)\chi_{(3,1)}+(q^2+q^4)\chi_{(2,2)}+(q^3+q^4+q^5)\chi_{(2,1,1)}+q^6\chi_{(1,1,1,1)}. 
\end{align*}

\end{example}

 Let $\la \vdash d, \ p = \la_1$. We realize the TGP module $R(\la)$ as a Zelevinsky tensor product as follows. Let $M(\la)$ be the $S_d$-module given by the Zelevinsky tensor product 
    $$ M(\la) = L(\la_1^t)\boxtimes L(\la_2^t)\boxtimes \cdots \boxtimes L(\la_p^t).$$

We define the integer $d_\la$ as the multinomial coefficient
\begin{equation*} \label{multinomial}
 d_\la = \binom{d}{\la^t_{1},\,\dots, \la^t_{ p}} = \frac{d!}{(\la_1^t)!\cdots (\la_p^t)!}.   
\end{equation*}
We have $\dim M({\la}) = d_\la$.
   \begin{lem}
   We have $R(\la) \simeq M(\la)$ as $S_d$-modules. In particular,  $\dim R(\la) = d_\la$.
   \end{lem}
   
   \begin{proof}
       The $S_d$-module $M({\la})$ is isomorphic to the Specht module of $\la^t$ constructed as the induced module from the trivial  representation of the subalgebra $\mathbb C S_{\la_1^t}\otimes \cdots \otimes \mathbb C S_{\la_p^t}\hookrightarrow \mathbb C S_d$ by
  $$
  M({\la}) \simeq \on{Ind}^{\mathbb C S_d}_{\mathbb C S_{\la_1^t}\otimes \cdots \otimes \, \mathbb C S_{\la_p^t}}(\mathbb C). 
  $$ 
The $S_d$-module $M({\la})$ has 
the following $S_d$-decomposition, see Section 7.2 of \cite{F},
$$
M({\la}) \simeq L(\la^t)\oplus \big (\bigoplus_{\mu \,\triangleright \la^t} L(\mu)^{\oplus K_{\mu,\la^t}}\big).
$$
Thus, by Theorem \ref{GP thm}, we have $\ch \, M({\la}) = \ch\, R(\la).$
   \end{proof}

\subsection{Deformed Garsia-Procesi modules}
\label{deformed TGP}

Let $\la \vdash d$ and $\bs a = (a_1,\dots, a_{\la_1})$ be a sequence of evaluation parameters of $\la$. Recall that we regard $a_i$ as labels for the columns of the Young diagram of $\la$. We also label by $a_i$ all the boxes in the $i$-th column of the Young diagram of $\la$.

For $0\leq n < \ell(\lambda)$, let $\boldsymbol {a}^{(n)} $ be the $m_\lambda(n)$-tuple of complex numbers obtained by writing the evaluation parameters $a_i$  labeling the boxes in the rows numbered by  $n+1, \dots, \ell(\lambda)$, from left to right and from top to bottom
\begin{equation} \label{param}
    \boldsymbol {a}^{(n)} = (a_1,\dots, a_{\lambda_{n+1}}, a_1,\dots, a_{\lambda_{n+2}},\dots, a_1,\dots,a_{\lambda_{\ell(\lambda)}}).
\end{equation}

 Let $h_k\big(\boldsymbol {a}^{(n)}\big) \,{\in}\, \mathbb C$ be the result of evaluating the $k$-th complete homogeneous symmetric polynomial  at the $m_\la(n)$-tuple $\boldsymbol {a}^{(n)}$. For $0\leq n < \ell(\la)$, let 
$$
{\mathscr C}^{\,\boldsymbol a}_{\lambda,\, n} = \bigg \{\sum_{k=0}^r (-1)^ke_{r-k}(t_J)h_{k}\big( \boldsymbol {a}^{(n)}\big) :  J\in \mathbb J^{(d-n)}, \ \  d-n-m_\lambda(n)<r\leq d-n \bigg\}.
$$
We define the set of deformed Tanisaki  polynomials ${\mathscr C}^{\, \boldsymbol a}_\lambda$ by 
$${ \mathscr C}^{\, \boldsymbol a}_\lambda = \bigcup_{n = 0}^{\ell(\lambda)-1}  { \mathscr C}^{\,\boldsymbol a}_{\lambda,\,n}.$$

Note that, for $\bs 0=(0,\dots,0)$, the set of deformed Tanisaki polynomials is the set of Tanisaki polynomials previously defined, i.e.,  ${\mathscr C}^{\, \boldsymbol 0}_\lambda={\mathscr C}_\lambda$. 

Let $\langle {\mathscr C}^{\, \boldsymbol a}_\lambda\rangle \subseteq \mathbb C[t_1,\dots,t_d]$ be the ideal generated by ${ \mathscr C}^{\,\boldsymbol a}_\lambda$.  We call the quotient 
$$R_{\bs a}(\lambda) = \mathbb C[t_1,\dots,t_d]/ \langle {\mathscr C}^{\, \boldsymbol a}_\lambda \rangle  $$
the deformed Tanisaki-Garsia-Procesi (deformed TGP) algebra $R_{\bs a}(\lambda)$. In the case of a single evaluation parameter $\bs a = (a,\dots ,a)$, we simply denote by $R_a(\la)$ the deformed TGP algebra $R_{\bs a}(\la)$.

\medskip

It is convenient to rewrite the relations in $R_{\bs a}(\la)$ in the following way.

\begin{lem} \label{relations lemma}
Let $0\leq n < \ell(\la).$ The set $\mathscr C_{\la,n}^{\,\bs a}$ is the union over  $J\in \mathbb J^{(d-n)}$ of the sets of 
coefficients of $x^{-1},\dots, x^{-m_\la(n)}$ in the Laurent series $H_{\la,n}^{\bs a}(t_J)$  in a formal variable $x^{-1}$ given by
\begin{equation}\label{series}
H_{\la,n}^{\bs a}(x;t_J)=\frac{\prod_{j\in J} (x+t_j)}{ \prod_{i=1}^{m_\la(n)}(x+\bs a^{(n)}_i)}.
\end{equation}
Moreover,  the Laurent series $H_{\la,n}^{\bs a}(x; t_J)$  truncates in $R_{\bs a}(\la)$ to a polynomial of degree $d-n-m_\la(n)$.
\end{lem}
\begin{proof}
Notice that $$H_{\la,n}^{\bs a}(x;t_J) = \sum_{r=0}^\infty \bigg( \sum_{k=0}^r(-1)^ke_{r-k}(t_J)h_k(\bs a^{(n)})\bigg)x^{d-n-m_\la(n)-r}.$$

We only need to show that $H_{\la,n}^{\bs a}(x;t_J)$ is a polynomial of degree $d-n-m_\la(n)$ in $R_{\bs a}(\la)$. Since the coefficients of $x^{-1},\dots, x^{-m_\la(n)}$  in $H_{\la,n}^{\bs a}(x;t_J)$ vanish in $R_{\bs a}(\la)$, then
\be
H_{\la,n}^{\bs a}(x; t_J) = f(x)+x^{-(m_\la(n)+1)}g(x),
\ee
where $f(x)$ is a polynomial of degree $d-n-m_\la(n)$ and $g(x)$ is a power series in $x^{-1}$.
Multiplying by the denominator of $H_{\la, n}^{\bs a}(x;t_J)$, we obtain that the product
\be
\big(f(x)+x^{-(m_\la(n)+1)}g(x)\big)\cdot\prod_{i=1}^{m_\la(n)}(x+\bs a^{(n)}_i)
\ee 
is a polynomial of degree $d-n$. As $f(x)\cdot \prod_{i=1}^{m_\la(n)}(x+\bs a^{(n)}_i)$ is also a polynomial of degree $d-n$, it follows that $g(x)=0$. 
\end{proof}
      
 \medskip

One can reduce the number of generators in the ideal $\langle  {\mathscr C}^{\, \boldsymbol a}_\lambda \rangle$ as follows.

\begin{lem}\label{reduce lem}
The ideal $\langle  {\mathscr C}^{\, \boldsymbol a}_\lambda \rangle $ is generated by the coefficients of $x^{-1}, \dots, x^{-\la_{n+1}}$ in $H_{\la,n}^{\bs a}(x;t_J)$ with $0\leq n<\ell(\la)$, and $J\in \mathbb J^{(d-n)}$.
\end{lem}
\begin{proof}
We prove that given the relations in the lemma, the functions $H_{\la,n}^{\bs a}(x;t_J)$ are polynomials by inverse induction on $n$. For $n=\ell(\la)$ the statement is trivial. Note that $\bs a^{(n+1)}$ is a subsequence of  $\bs a^{(n)}$. Using the statement for $n+1$ and the given relations we obtain 
\be
H_{\la,n}^{\bs a}(x;t_J)=\frac{\prod_{j\in J} (x+t_j)}{ \prod_{i=1}^{m_\la(n)}(x+\bs a^{(n)}_i)}=\frac{f(x)}{\prod_{i=1}^{\la_{n+1}}(x+\bs a^{(n)}_i)}=g(x)+x^{-\la_{n+1}-1}h(x),
\ee 
where $f(x), g(x)$ are polynomials and $h(x)$ is a series in $x^{-1}$. Multiplying the last equality by the ${\prod_{i=1}^{\la_{n+1}}(x+\bs a^{(n)}_i)}$, we see that $h(x)=0$.
\end{proof}
However, the number of relations  in general is still much larger than $d$.  

\begin{example}
   Let $d=3$ and $\la=(2,1)$ and $a_1=a_2=a$, cf. Example \ref{example} . Then, we have $6$ deformed Tanisaki polynomials in $\mathscr C_\la^{\,a}$ ($3$ for $n=0$  and $3$ for $n=1$). According to Lemma \ref{reduce lem}, one can reduce the number of generators of $\langle \mathscr C_\la^{\,a} \rangle$ to $5$ polynomials ($2$ for $n=0$ and $3$ for $n=1$). In fact, one can reduce it further to $4$ polynomials: $t_1+t_2+t_3-3a, \   (t_1-a)(t_2-a), \ (t_2-a)(t_3-a),\ (t_1-a) (t_3-a)$.  It is easy to see that it cannot be reduced to $3$ generators.
\end{example}

\medskip

We will show below that our deformation is flat, i.e., 
$\dim R_{\bs a}(\la) = \dim R(\la)$ for all $\bs a$. We start by showing $R_{\bs a}(\la)$ is finite-dimensional and establishing a few simple properties.

\begin{lem}  \label{symmev}
    Let $\la\vdash d.$ In $R_{\bs a}(\la),$ the relation $s(t_1,\dots, t_d)  = s(\bs a^{(0)})$ holds for any symmetric polynomial $s(t_1,\dots,t_d)\in \mathbb C[t_1,\dots,t_d]^{S_d}.$ In particular, $\dim(R_{\bs a}(\lambda))\leq d!$.
\end{lem}
\begin{proof} By Lemma \ref{relations lemma} applied to $J=\{t_1,\dots,t_d\}$, in $R_{\bs a}(\la)$ we have
\be
\prod_{i=1}^d (x+t_i) =  \prod_{i=1}^{d}(x+\bs a^{(0)}_i).
\ee
Therefore, we have $e_i(t_1,\dots,t_d)=e_i(\bs a^{(0)})$ in $R_{\bs a}(\lambda)$. The lemma now follows, since $e_i(t_1,\dots,t_d)$ generate the algebra of symmetric polynomials.
\end{proof}

The next corollary will be useful later.

\begin{cor} 
    Let $\la \vdash d$. In $R_{\bs a}(\la)$, we have \begin{equation*}
        \prod_{j=1}^d(t_i-\bs a_j^{(0)}) = 0,\qquad 1\leq i \leq d. \label{t-a} 
    \end{equation*}
In particular, if $\bs a_j^{(0)}\neq 0$ for all $j$,  the operator $t_i^{-1}$ in $R_{\bs a}(\la)$ can be expressed as a polynomial in the operator $t_i$.
 \end{cor}
\begin{proof}
By Lemma \ref{symmev}, we have 
\be
\prod_{j=1}^d(x-t_j)=\prod_{j=1}^d(x-\bs a_j^{(0)}).
\ee
By substituting $x=t_i,$ the result follows.
\end{proof}

The next simple lemma asserts that the algebra $R_{\bs a}(\la)$ is invariant under shifting and scaling of evaluation parameters, as well as swaps of parameters corresponding to columns of the same length.
\begin{lem} Let $\la\vdash d$. 
\begin{enumerate}  \label{swapshift}
    \item Let $\bs a=(a_1,\dots,a_{\la_1}),   \widetilde{\bs a}=(ba_1+c,\dots, ba_{\la_1}+c)$, where $b,c\in\C$ with $b\neq 0$. Then, there is an isomorphism of algebras $R_{\widetilde{\bs a}}(\la)\to R_{\bs a}(\la)$ which maps 
    $t_k\mapsto bt_k+c$. 
\item Let $\la^t_i=\la^t_j$ for some $i,j$ and let $\bs a=(a_1,\dots, a_{\la_1}), \widetilde{\bs a}=(a_1,\dots, a_j,\dots,a_i,\dots, a_{\la_1})$. Then, there is an isomorphism of algebras $R_{\widetilde{\bs a}}(\la) \to R_{\bs a}(\la)$ which maps $t_k\mapsto t_k$. 
\end{enumerate}

\end{lem}
\begin{proof}
For (i), we have the assignment
\be
H_{\la,n}^{\widetilde{\bs a}}(x;t_J)  =\frac{\prod_{j\in J}(x+t_j)}{\prod_{i=1}^{m_\la(n)}(x+\widetilde{\bs a}^{(n)}_i)} \mapsto 
 \frac{\prod_{j\in J} (x+bt_j +c)}{\prod_{i=1}^{m_\la(n)}(x+b\bs a^{(n)}_i+c)}=b^{d-n-m_\la(n)} \frac{\prod_{j\in J} (u+t_j)}{\prod_{i=1}^{m_\la(n)}(u+\bs a^{(n)}_i)},
\ee
where $u=(x+c)/b$. The last expression is a polynomial in
 $R_{\bs a}(\la)$, therefore the map is well-defined.

For (ii), note that the set of coordinates of $\bs a^{(n)}$ and $\widetilde{\bs a}^{(n)}$  is the same for all $0\leq n<\ell(\la)$. 
\end{proof}

\medskip

We define a structure of $\widetilde S_d$-module on $R_{\bs a}(\la)$.

\begin{lem}
Let $a_i\neq 0$ for all $i$. There is an right action of  $\widetilde S_d$ on $R_{\bs a}(\lambda)$ such that $\sigma_i\in \widetilde S_d$ act by permuting the variables $t_1,\dots,t_d$ and $t_i\in \widetilde S_d$ act as  the multiplication operator. 
\end{lem}
\begin{proof}

The ideal $\langle {\mathscr C}^{\, \boldsymbol a}_\lambda\rangle$ is symmetric, so the deformed TGP algebra is naturally a right $S_d$-module. The remaining relations for the generators of $\widetilde S_d$ are trivially satisfied. If $a_i\neq 0$ for all $i$, then the operators $t_i$ are invertible in $R_{\bs{a}}(\lambda)$ by Corollary \ref{t-a}.
\end{proof}

We call the right $\widetilde S_d$-module  $R_{\bs a}(\la)$ the deformed TGP module $R_{\bs a}(\la)$. In particular, in such a case we implicitly assume that $a_i\neq 0$ for all $i$. We regard the $1$-dimensional $S_d$-module $L(1^{\{d \}})$ as an $\widetilde S_d$-module by letting $t_i$ act by one. In this way, there is a natural structure of right $\widetilde S_d$-module defined in the tensor product 
$$
\widetilde R_{ \bs a}(\la) = R_{ \bs a}(\la)\otimes L(1^{\{d \}}).
$$
We call the right $\widetilde S_d$-module $\widetilde R_{\bs a}(\la)$ the amended deformed TGP module $\widetilde{R}_{\bs a}(\la)$. 

We note that $\widetilde R_{ \bs a}(\la)=R_{ \bs a}(\la)$ as $\C[t_1,\dots,t_d]$-modules.

\section{Deformed TGP modules and Weyl modules}  \label{main sec}

\subsection{The dimension of  \texorpdfstring{$R_{\boldsymbol a}(\lambda)$}{} }

Let $\la \vdash d$ and $p = \la_1$. Given a sequence of evaluation parameters $\bs a = (a_1,\dots, a_p)$, let $M_{\bs a}(\la)$ be the $\widetilde S_d$-module defined by the affine Zelevisnky tensor product
    $$ M_{\bs a}({\la}) = L_{a_1}(\la_1^t)\widetilde\boxtimes L_{a_2}(\la_2^t)\widetilde\boxtimes \cdots \widetilde\boxtimes L_{a_p}(\la_p^t).$$
    Note that the $S_d$-structure of $M_{\bs a}(\la)$ is independent of the choice of $\bs a$, i.e., we have $M_{\bs a}(\la) \simeq M(\la)$ as $S_d$-modules for all $\bs a$.

We will show that the dimension of $R_{\bs a}(\la)$ does not depend on $\bs a$ either. We prove this claim by showing the following chain of inequalities
\begin{align}\label{ineq  chain}
d_\la=\dim R(\la) \geq \dim R_{\epsilon \bs a}(\la) =\dim R_{\bs a}(\la)\geq \dim R_{\bs a+\epsilon \bs s}(\la)\geq \dim M(\la) = d_{\la}, 
\end{align}
where $\bs s=(0,1,2,\dots,p-1)$ and $\epsilon>0$ is sufficiently small.  

\medskip 
 The following lemma is used in proving the rightmost inequality.

\begin{lem}\label{annihilation lem}
For any $f(t_1,\dots,t_d)\in \langle \mathscr C_\la^{\,\bs a}\rangle $, we have $f\cdot M_{\bs  a}({\la}) = 0$.
\end{lem}
\begin{proof}
Let $m = (m_1\otimes \cdots \otimes m_p)\otimes 1 \in M_{\bs a}({\la})$, where $m_i$ is a choice of basis vector for the trivial 1-dimensional  $S_{\la_i^t}$-module $L(\la_i^t).$ From the action of $\widetilde S_d$ on $M_{\bs a}({\la})$ given by (\ref{affzel}), we have
    $$m\cdot t_k = a_i\, m \qquad \mbox{for} \quad \la_1^t+\dots +\la_{i-1}^t < k \leq \la_1^t+\dots + \la_{i-1}^t +\la_i^t.  $$
    
We label by $a_i$ each box in the $i$-th column of the Young diagram of $\la$ just as in Section \ref{deformed TGP}. For $0\leq n<\ell(\la)$, let $\bs a^{(n)}$ be the sequence  defined in (\ref{param}). For  $J = \{t_{i_1},\dots, t_{i_{d-n}}\}\in \mathbb J^{(d-n)}$, let $\bs b^{(J)} = (b_1,\dots,b_{d-n}) \in \mathbb C^{d-n}$ be the sequence of complex numbers defined by
$$m\cdot t_{i_k} = b_k m.$$

Denote by $e_r(\bs b^{(J)})\,{\in}\, \mathbb C$ the scalar obtained by evaluating the elementary symmetric polynomial $e_r(t_J)$ on $\bs b^{(J)},$ so that
$ m\cdot e_r(t_J) = e_r(\bs b^{(J)})m$. Then, 
\begin{equation}
\label{c cdot}
m\cdot H_{\la,n}^{\,\bs a} (x; t_J) = m\Bigg( \frac{\prod_{i=1}^{d-n}(x+\bs b_i^{(J)})}{\prod_{i=1}^{m_\la(n)} (x+\bs a_i^{(n)})}\Bigg). \end{equation}

We show that the coefficients of $x^{-1}, \dots, x^{-m}$ corresponding to the right-hand side of $\eqref{c cdot}$ are zero. Indeed, if $a_i$ occurs a total of $s>0$ times as a coordinate in $\bs a^{(n)}$,  then $\la^t_i\geq n+s$ and $a_i$ occurs at least $s$ times in $\bs b^{(J)}$. As a result, the set of coordinates of $\bs a^{(n)}$ is a subset of the set of coordinates of $\bs{b}^{(J)}$. Therefore, the right-hand side of \eqref{c cdot} is a polynomial and the coefficients of $x^{-1}, \dots, x^{-m}$ must be zero. It follows that  $\langle \mathscr C_{\la,n}^{\,\bs a}\rangle$ annihilates the vector $m\in M_{\bs a}(\la).$ 

Since  $\langle \mathscr C_\la^{\,\bs a} \rangle $ is a symmetric ideal,  then  $\langle \mathscr C_\la^{\,\bs a} \rangle $ annihilates all the vectors of the form $m\cdot \sigma, \sigma\in S_d,$ proving the result.
\end{proof}

\begin{cor}\label{distinct cor}
    If $a_i\neq a_j$ for all $i\neq j$, there exists a surjective homomorphism of $\widetilde S_d$-modules
$$  R_{\bs a}(\la)  \twoheadrightarrow M_{\bs a}({\la}).$$ 
In particular, we have $\dim  R_{\bs a}(\la) \geq d_\la.$
\end{cor}
\begin{proof}

If $a_i\neq a_j$ for all $i\neq j$, the $\widetilde S_d$-module $M_{\bs a}(\la)$ is irreducible and, in particular, cyclic. 
Since by Lemma \ref{annihilation lem} the ideal $\langle\mathscr C_\la^{\,\bs a}\rangle$ annihilates $M_{\bs a}(\la),$ there exists a surjective homomorphism of $\widetilde S_d$-modules
$R_{\bs a}(\la)  \twoheadrightarrow M_{\bs a}(\la).$
\end{proof}
\medskip
Now we are ready to prove our first main result.

\begin{thm}\label{flat thm} Let $\la \vdash d$ and let $R_{\bs a}(\la)$ be the deformed TGP algebra.

\begin{enumerate}
    \item We have $R_{\bs a}(\la)\simeq M(\la)$ as $S_d$-modules.  In particular, $\dim R_{\bs a}(\la) = d_\la$.  \label{item 1}
    \item If all $a_i$ are non-zero and distinct, then the deformed TGP module $R_{\bs a}(\la)$ is irreducible and $R_{\bs a}(\la) \simeq  M_{\bs a}(\la)$ as $\widetilde S_d$-modules.
\end{enumerate}

\end{thm}
\begin{proof}
We finally show the chain of inequalities \eqref{ineq chain}.  By Lemma \ref{swapshift}, for any $\epsilon>0$, there exists an algebra isomorphism between $R_{\bs a}(\la)$ and   $R_{\epsilon \bs a}(\la)$, so $\dim R_{\bs a}(\la) = \dim R_{\epsilon \bs a}$.

The inequalities $\dim R(\la)\geq \dim R_{\epsilon \bs a}(\la)$  and $\dim R_{\bs a}(\la)\geq \dim R_{\bs a+\epsilon \bs s}(\la)$ for sufficiently small $\epsilon$ both follow from  Lemma \ref{deformed ideal lemma}. Again, for $\epsilon>0$ small enough, the sequence of evaluation parameters $\bs a+\epsilon \bs s$ has distinct coordinates, thus the inequality $\dim R_{\bs a {+\epsilon \bs s}}(\la)\geq  \dim M_{\bs a +\epsilon \bs s}(\la) = \dim M(\la) $ follows from Corollary \ref{distinct cor}.

It follows that all inequalities in \eqref{ineq  chain} are in fact equalities and, in particular,   $\dim R_{\bs a}(\la) = d_\la$. Therefore, the isomorphism of $S_d$-modules $R_{\bs a}(\la)\simeq R(\la)$ follows from Lemma \ref{deformed module lemma}, proving \eqref{item 1}.

For (ii), under the hypothesis $a_i$ non-zero and $a_i\neq a_j$, it follows that the map in Corollary \ref{distinct cor} is an isomorphism, since the dimensions of $R_{\bs a}(\la)$ and of $M_{\bs a}(\la)$ are equal.
\end{proof}

\subsection{Zelevinsky tensor products of TGP modules} Let $d=d_1+d_2$ and let $\kappa : \ \C[t_1,\dots,t_{d}]\to \C[t_1,\dots,t_{d_1}]\otimes \C[t_1,\dots,t_{d_2}]$ be the natural isomorphism of algebras given by $t_i \mapsto t_i\otimes 1 $ for $ 1\leq i \leq d_1$ and  $t_i\mapsto 1 \otimes t_i$ for  ${d_1+}1\leq i \leq d.$

For $\la\vdash d_1, \ \mu\vdash d_2$  and $\bs a=(a_1,\dots,a_{\la_1}), \ \bs b=(b_1,\dots, b_{\mu_1})$,  let $\nu=\la\sqcup\mu$ and   $\bs c=\bs a \sqcup \bs b$ be their concatenated partition and sequence of evaluation parameters.

\begin{lem} \label{isom}
 The map $\kappa:\mathbb C[t_1,\dots, t_{d_1+d_2}]\to \mathbb C[t_1,\dots, t_{d_1}]\otimes \mathbb C[t_1,\dots, t_{d_2}]$ descends to a surjective homomorphism of algebras $\kappa :\ R_{\bs c}(\nu) \to R_{\bs a}(\la) \otimes R_{\bs b}(\mu)$.
\end{lem}
\begin{proof}
We only need to check that the relations in $R_{\bs c}(\nu)$ are mapped to zero. We have
\be
H_{\nu,n}^{\bs c}(x;t_J) \mapsto H_{\la,n_1}^{\bs a}(x;t_{J_1})\otimes H_{\mu,n_2}^{\bs b}(x; t_{J_2}), 
\ee 
where $J_1 = J\cap \{t_1,\dots, t_{d_1}\}$, $J_2 = J\cap \{t_{d_1+1},\dots, t_{d_1+d_2}\}$ and $n_1  = d_1-|S_1|$, $n_2 = d_2-|S_2|$. Note that $n_1+n_2=n$.

In $R_{\bs a}(\la) \otimes R_{\bs b}(\mu)$,  the series $H_{\la,n_1}^{\bs a}(x;t_{J_1})\otimes 1$ and  $1\otimes H_{\mu,n_2}^{\bs b}(x; t_{J_2}) $ are polynomials. 
\end{proof}

 We denote by $1_{\bs a}, 1_{\bs b}$  the cyclic vectors of  the deformed TGP modules $ R_{\bs a}(\la),  R_{\bs b}(\mu)$, respectively.

\begin{thm} \label{breakdown}
The map $\mathbb C[t_1,\dots, t_d] \to R_{\bs a}(\la)\widetilde \boxtimes R_{\bs b}(\mu)$ defined by
\begin{equation} \label{mapsto}
    f\mapsto \sum_{\sigma\in S_{d}} 1_{\bs a}\otimes 1_{\bs b} \otimes\, \sigma f,
\end{equation} 
descends to a homomorphism of $\widetilde S_d$-modules   $R_{\bs c}(\nu)\rightarrow  R_{\bs a}(\lambda)\widetilde\boxtimes  R_{\bs b}(\mu)$. Moreover,  if $a_i\neq b_j$ for all $i,j$, this map is an isomorphism of $\widetilde S_d$-modules. 
\end{thm}

\begin{proof} 
We first check that the relations in $R_{\bs c}(\nu)$ are mapped to zero. For any $\sigma \in S_d$ and $f\in \langle \mathscr C_\nu^{\,\bs c}\rangle$, the polynomial  $\widetilde f  = f\cdot \sigma^{-1} \in \langle \mathscr C_\nu^{\,\bs c} \rangle$. Thus, it follows  from  Lemma \ref{isom} that
\begin{equation*}
    \sum_{\sigma\in S_d} 1_{\bs a}\otimes 1_{\bs b}\otimes \sigma f = \sum_{\sigma \in S_d} (1_{\bs a} \otimes 1_{\bs b})\cdot \widetilde f \otimes  \sigma=0.
\end{equation*}
Moreover, using \eqref{semidirect}, one verifies that \eqref{mapsto}  descends to a homomorphism of $\widetilde S_d$-modules $R_{\bs c}(\nu) \mapsto R_{\bs a}(\la)\widetilde\boxtimes R_{\bs b}(\mu)$.

We now show that \eqref{mapsto} is surjective under the hypothesis  $a_i\neq b_j$ for all $i,j$. Let
$$h = \bigg(\prod_{i = 1}^{d_1}\prod_{j=1}^{d_2}(t_i-\bs b_j^{(0)})\bigg) \bigg(\prod_{i = d_1+1}^{d_1+d_2}\prod_{j=1}^{d_1} (t_i-\bs a^{(0)}_j)\bigg)\in  R_{\bs c}(\nu). $$
If $\sigma \not \in S_{d_1}\times S_{d_2}\hookrightarrow S_d$, then $1_{\bs{a}}\otimes 1_{\bs{b}}\otimes \sigma h=0$ by Corollary \ref{t-a}. Therefore, we have
\be
h\mapsto \sum_{\sigma \in S_d}1_{\bs a}\otimes 1_{\bs b} \otimes \sigma  h = \gamma \ ( 1_{\bs a}\otimes 1_{\bs b} \otimes 1),
\ee
where $\gamma$ is the non-zero scalar
$$\gamma = (d_1!)(d_2!) \prod_{i=1}^{d_1} \prod_{j=1}^{d_2}(\bs a_i^{(0)}-\bs b_j^{(0)})(\bs b_j^{(0)}-\bs a_i^{(0)}). $$ 

Since $1_{\bs a}\otimes 1_{\bs b}\otimes 1$ is a cyclic vector of $R_{\bs a}(\lambda)\widetilde\boxtimes  R_{\bs b}(\mu)$, the map is surjective. As the dimensions of  $R_{\bs c}(\nu)$ and $R_{\bs a}(\la)\widetilde \boxtimes R_{\bs b}(\mu)$ coincide by 
Theorem \ref{flat thm} and 
\eqref{forgetzele}, then  \eqref{mapsto} is an isomorphism.
\end{proof}

As a result, we obtain that the TGP module $ R_{\bs b}(\la)$ is the Zelevinsky tensor product of the TGP modules corresponding to the splitting of $\la$ and $\bs b$ and similarly for the amended TGP module $\widetilde R_{\bs b}(\la)$.

\begin{cor}\label{split of R_b cor}
Let $\la = \la^{(1)}\sqcup \cdots \sqcup \la^{(s)}$ and $\ \bs b = \bs a_1\sqcup \cdots \sqcup \bs a_s$ be the splitting of $\la$ and $\bs b$ with $\bs a_i = (a_i,\dots,a_i)$. We have the isomorphisms of $\widetilde S_d$-modules
$$R_{\bs b}(\la) \simeq R_{a_1}(\la^{(1)})\widetilde \boxtimes \cdots \widetilde \boxtimes R_{a_s}(\la^{(s)}), \qquad \widetilde R_{\bs b}(\la) \simeq \widetilde R_{a_1}(\la^{(1)})\widetilde \boxtimes \cdots \widetilde \boxtimes \widetilde R_{a_s}(\la^{(s)}).$$
\end{cor}

\begin{proof}
    The corollary follows from Theorem \ref{breakdown} by induction on $s$ and the associativity of the affine Zelevinsky tensor product.
\end{proof}

\subsection{Amended deformed TGP modules and local Weyl modules.}  \label{amendedlocal}

In this subsection, we study the image of the amended deformed TGP modules  under the affine Schur-Weyl duality functor. 

\begin{lem} \label{soclehead}
    Let $\lambda\vdash d$ and $a\in \mathbb C^\times$. The amended deformed TGP module $\widetilde R_a(\la)$ has head $L_a(1^{\{ d \}})$ and socle containing $L_a(\la)$. 
\end{lem}
\begin{proof} Recall that $\widetilde R_{a}(\la)= R_{a}(\la)\otimes L(1^{\{ d \}})$. Thus, by \eqref{transpose}, it suffices to show that $R_{a}(\la)$ has head $L_a(d)$ and socle containing $L_a(\lambda^t)$.

 By construction, the deformed TGP module $R_{a}(\la)$ has a cyclic vector $1\in R_a(\la)$ which generates a trivial $S_d$-module $\mathbb C  S_d\cdot  1 \simeq L(d)$. Moreover, $[R_{a}(\la): L(d)] = 1$. Therefore, any proper submodule of $R_{a}(\la)$ is a direct sum of irreducible $S_d$-modules which do not contain the trivial representation. It follows that $R_a(\la)$  has a unique maximal submodule and the head of $R_a(\la)$ is $L_a(d)$.

    From Theorem \ref{GP thm} and the isomorphism  of $\mathbb C S_d\otimes \mathbb C[t_1,\dots,t_d]$-modules $R_a(\la)\simeq R(\la)$, we know that the graded component of maximal degree of $R_a(\la)$ is isomorphic to $L(\la^t)$ as $S_d$-modules. Therefore, the socle of $R_a(\la)$ contains $L_a(\la^t)$. 
\end{proof}

Recall the affine Schur-Weyl duality functor $\widetilde{\mathcal F}_d: \mbox{Rep}(\widetilde S_d)\to \mbox{Rep}_d(\widetilde \g)$, see Theorem \ref{affineschur}.

\begin{thm} \label{dualweyl} Let $\lambda\vdash d, \ a\in \mathbb C^\times$. 
 We have the isomorphism of $\widetilde \g$-modules $\widetilde{\mathcal F}_d(\widetilde R_a(\la))\simeq W_{a^{-1}}(\la)^{\vee}.$    
\end{thm}
\begin{proof}

From Theorem \ref{flat thm}, we have the following isomorphism of $S_d$-modules 
$$\widetilde R_a(\la)\simeq M(\la)\otimes L(1^{\{ d \}}) \simeq  L(1^{\{ \la_1^t\}}) \boxtimes \cdots \boxtimes L(1^{\{ \lambda_{p}^t\}}),$$ 
where $p=\la_1$.
Thus, by \eqref{functorzele} and Theorem \ref{localweyl} (ii), we have the isomorphism of $\g$-modules
    $$\widetilde{\mathcal F}_d(\widetilde R_a(\la))\simeq
    V(\omega_{\la_1^t})\otimes \dots \otimes  V(\omega_{\la_p^t})
   \simeq V(\omega_1)^{\otimes \lambda(h_1)}\otimes \cdots \otimes V(\omega_d)^{\otimes \lambda(h_d)} \simeq W_{a^{-1}}(\la),$$
    where $d=|\la|$.  By Lemma \ref{soclehead}, the $\widetilde \g$-module $\widetilde{\mathcal F}_d(\widetilde R_a(\la))$ has head $V_a(\omega_d)$ and socle containing $V_a(\la)$. Therefore, from \eqref{veedual} we see that $(\widetilde{\mathcal F}_d(\widetilde R_a(\la))^{\vee}$ satisfies the hypothesis of Theorem \ref{localweyl} (v) where $a$ is replaced by $a^{-1}$, and therefore is isomorphic to the Weyl module $W_{a^{-1}}(\la)$.
\end{proof}

Now, we arrive at the last main result of the paper.
\begin{cor} \label{dual} For $\la \vdash d$ and $\bs b=(b_1,\dots,b_{\la_1})$, $b_i\in\C^\times$,  let  $\la = \la^{(1)}\sqcup \cdots \sqcup \la^{(s)}, \ \bs b = \bs a_1\sqcup \cdots \sqcup \bs a_s$  be the splitting of $\la$ and $\bs b$ with $\bs a_i = (a_i,\dots, a_i).$ We have an isomorphism of $\widetilde \g$-modules
$$\widetilde{\mathcal F}_d(\widetilde R_{\bs b}(\la)) \simeq \big( W_{a_1^{-1}}(\la^{(1)})\otimes \cdots \otimes W_{a_s^{-1}}(\la^{(s)})\big)^{\vee}. $$
\end{cor}

\begin{proof} 
The corollary follows from Corollary \ref{split of R_b cor}, Theorem \ref{dualweyl} and formula \eqref{functortensor}.
\end{proof}

\section{An example}\label{example sec}
We illustrate the situation with an explicit example.
\begin{example} \label{example}
    Let $\lambda = (2,1)$. For $a,b\in \mathbb C^\times$, set $M =  L_a(1,1) \widetilde\boxtimes L_b(1).$ Let $\rho: \widetilde S_3 \to \mbox{GL}(M)$ be the associated representation. 
   If $a\neq b$, then $M$ is irreducible and isomorphic to the amended deformed TGP module $M\simeq \widetilde R_{(a,b)}(\lambda)$.

    The right $\widetilde S_3$-module $M$ has a basis $\alpha = \{v_0, v_1, v_2\}$ given by 
    $$v_0 = m_1\otimes m_2 \otimes (1), \qquad v_1 = m_1 \otimes m_2 \otimes (2,3),\qquad v_2 =  m_1 \otimes m_2 \otimes (1,3),$$
    where $L_a(1,1)=\C m_1$ and $L_b(1)=\C m_2$.

We now describe three different $\widetilde S_3$-modules which we denote by $M_0,M_1$, and $M_2$, obtained as  limits of $M$ as $b\to a$. We observe that all these three modules are non-isomorphic, so the result depends on the way the limit is taken. In each case, we choose a basis, write the action explicitly in this basis and take the limit of the corresponding matrices. 

\medskip
    
In the basis $\alpha,$ we have the following expression for the action of the generators of $\widetilde S_3$

\begin{align*}
[\rho(1,2)]_\alpha  =     \begin{bmatrix}
 -1 & 0 & 0\\
0 & 0 & -1 \\
0 & -1 & 0
\end{bmatrix}&, \qquad  [\rho(2,3)]_\alpha =  \begin{bmatrix}
0 & 1 & 0\\
1 & 0 & 0 \\
0 & 0 & -1
\end{bmatrix},
\\\\
[\rho(t_1)]_\alpha = \begin{bmatrix}
a & 0 & 0\\
0 & a & 0\\
0 & 0 & b
\end{bmatrix}, \qquad [\rho(t_2)]_\alpha =& \begin{bmatrix} 
a & 0 & 0\\ 
0 & b & 0\\
0 & 0 & a
\end{bmatrix}, \qquad [\rho(t_3)]_\alpha = \begin{bmatrix}
b & 0 & 0\\
0 & a & 0\\
0 & 0 & a
\end{bmatrix}.
\\
\end{align*}

Let $M_0$ be the $ \widetilde S_3$-module with the basis $\alpha$ and action $\rho_0 : \widetilde S_3 \to GL(M_0)$ given by
$$[\rho_0(t_i)]_\alpha = \lim_{b\to a} [\rho(t_i)]_\alpha, \qquad [\rho_0(1,2)]_\alpha = [\rho(1,2)]_\alpha, \qquad [\rho_0(2,3)]_\alpha = [\rho(2,3)]_\alpha.$$ 
Clearly, we have the decomposition of  $ \widetilde S_3$-modules
 \begin{align*}
 &M_0\simeq L_a(1,1,1)\oplus L_a(2,1), \\ L_a(1,1,1) =  \mbox{ span}&(v_0-v_1-v_2), \qquad L_a(2,1) =  \mbox{span}(v_0+v_1, \, v_0+v_2). 
 \end{align*}

\medskip

Next, consider the basis $\beta = \{w_0, w_1, w_2\}$ of $M$ given by
$$w_0 = v_0-v_1-v_2,\qquad w_1 = w_0\cdot t_1, \qquad w_2 = w_0\cdot t_2. $$
In the basis $\beta$, we have 
\begin{align*}
[\rho(1,2)]_\beta  =     \begin{bmatrix}
 -1 & 0 & 0\\
0 & 0 & -1 \\
0 & -1 & 0
\end{bmatrix}&, \qquad  [\rho(2,3)]_\beta =  \begin{bmatrix}
-1 & 0 & -(2a+b)\\
0 & -1 & 1 \\
0 & 0 &  1
\end{bmatrix},
\\\\
[\rho(t_1)]_\beta = \begin{bmatrix}
0 & -ab & -a^2\\
1 & a+b & a\\
0 & 0 & a
\end{bmatrix}, \ [\rho(t_2)]_\beta =& \begin{bmatrix} 
0 & -a^2 & -ab\\ 
0 &  a & 0\\
1 & a & a+b
\end{bmatrix}, \  [\rho(t_3)]_\beta = \begin{bmatrix}
2a+b & a^2+ab & a^2+ab\\
 -1& 0 & -a\\
 -1& -a & 0
\end{bmatrix}.
\\
\end{align*}

Let $M_1$ be the $ \widetilde S_3$-module with the basis $\beta$ and action $\rho_1: \widetilde S_3 \to GL(M_1)$ given by
$$[\rho_1(t_i)]_\beta = \lim_{b\to a} [\rho(t_i)]_\beta, \qquad [\rho_1(1,2)]_\beta = [\rho(1,2)]_\beta, \qquad [\rho_1(2,3)]_\beta = \lim_{b\to a}[\rho(2,3)]_\beta.$$
Then, as $S_3$-modules, we have
\begin{align*}
 M_1\simeq &\ L(1,1,1)\oplus L(2,1), \\ L(1,1,1) = \mbox{span}(w_0), \qquad& L(2,1) =  \mbox{span}(-aw_0+w_1,-aw_0+ w_2). 
 \end{align*}
 
However, while $L{(2,1)}$ is an $\widetilde S_3$-submodule of $M_1$, we see that $L(1,1,1)$ is not. Moreover, $M_1$ is cyclic with cyclic vector $w_0$.
Thus, $M_1$ is an indecomposable module with submodule (the socle) $L_a{(2,1)}$ and quotient module (the head) $L_a(1,1,1)$. Thus,  $M_1$ is isomorphic  to the amended deformed TGP module $ \widetilde R_a(2,1)$ as $\widetilde S_3$-module.

\medskip

Finally, let $\gamma = \{u_0, u_1, u_2\}$ be the basis of $M$ given by
$$u_0 = v_0+v_1,\qquad u_1 = u_0\cdot (1,2), \qquad u_2 = u_0\cdot t_2. $$
With respect to the basis $\gamma$, we have
\begin{align*}
[\rho(1,2)]_\gamma  =     \begin{bmatrix}
 0 & 1 & b\\
1 & 0 & b \\
0 & 0 & -1
\end{bmatrix}&, \qquad  [\rho(2,3)]_\gamma =  \begin{bmatrix}
1 & -1 & a+b\\
0 & -1 & 0 \\
0 & 0 &  -1
\end{bmatrix},
\\\\
[\rho(t_1)]_\gamma = \begin{bmatrix}
a & b & 0\\
0 & b &    0   \\
0 & -1 &  a
\end{bmatrix}, \quad [\rho(t_2)]_\gamma =& \begin{bmatrix} 
0 & 0 & -ab\\ 
0 &  a & 0\\
1 & 0 & a+b
\end{bmatrix}, \quad [\rho(t_3)]_\gamma = \begin{bmatrix}
 a+b& -b & ab \\
 0&  a& 0\\
 -1&  1& 0 
\end{bmatrix}.
\\
\end{align*}

Let $M_2$ be the $ \widetilde S_3$-module with the basis $\gamma$ and action $\rho_2: \widetilde S_3 \to GL(M_2)$ given by
$$[\rho_2(t_i)]_\gamma = \lim_{b\to a} [\rho(t_i)]_\gamma, \qquad [\rho_2(1,2)]_\gamma = \lim_{b\to a}[\rho(1,2)]_\gamma, \qquad [\rho_2(2,3)]_\gamma = \lim_{b\to a}[\rho(2,3)]_\gamma.$$
We again have the $S_3$-module decomposition
\begin{align*}
 M_2&\simeq L(1,1,1)\oplus L(2,1), \\ L(1,1,1) = \mbox{span}&(-au_0+u_2), \qquad L(2,1) =  \mbox{span}(u_0,u_1). 
 \end{align*}
 
In this case, $L(1,1,1)$ is an $\widetilde S_3$-submodule and $L(2,1)$ is not. Thus, $M_2$ is an indecomposable module with submodule (the socle) $L_a{(1,1,1)}$ and  quotient module (the head) $L_a(2,1)$. We see that $M_2$ is the dual module of  $M_1$ with $a$ changed to $a^{-1}$. Moreover, $M_2$ is mapped under the affine Schur-Weyl duality functor to the local Weyl module $\widetilde{\mathcal F}_3(M_2)\simeq W_{a}(\lambda)$.

\medskip

The following picture summarizes the three different $\widetilde S_3$-modules obtained as limits of $M$.

\begin{figure}[h]
   \begin{tikzpicture} 
\node at (-0.7,-0.3) {$M_0 = $};   
\node at (0.15, 1.1) {$a$}; 
\node at (0.15, -0.3) {$\bigoplus$}; 
\draw (0,0.6)--(0,0.9);
\draw (0,0.9)--(0.3,0.9);
\draw (0.3,0.9)--(0.3,0.6);
\draw (0.3,0.6)--(0,0.6); 

\draw (0,0.3)--(0,0.6);
\draw (0,0.6)--(0.3,0.6);
\draw (0.3,0.6)--(0.3,0.3);
\draw (0.3,0.3)--(0,0.3); 

\draw (0,0)--(0,0.3);
\draw (0,0.3)--(0.3,0.3);
\draw (0.3,0.3)--(0.3,0);
\draw (0.3,0)--(0,0); 

\node at (0.1,-1.4) {$a$}; 
\draw (0,-0.6)--(0, -0.9);
\draw (0,-0.9)--(0.3,-0.9);
\draw (0.3,-0.9)--(0.3,-0.6);
\draw (0.3,-0.6)--(0,-0.6);

\draw (0.3,-0.6)--(0.6,-0.6);
\draw (0.6,-0.6)--(0.6,-0.9);
\draw (0.6,-0.9)--(0.3,-0.9);

\draw (0,-0.9)--(0,-1.2);
\draw (0,-1.2)--(0.3,-1.2);
\draw (0.3,-1.2)--(0.3,-0.9);

\node at (2.3,-0.3) {$M_1 = $};
\node at (3.15, 1.1) {$a$}; 
 \draw[->]  (3.15, -0.1)--(3.15,-0.5);
\draw (3,0.6)--(3,0.9);
\draw (3,0.9)--(3.3,0.9);
\draw (3.3,0.9)--(3.3,0.6);
\draw (3.3,0.6)--(3,0.6); 

\draw (3,0.3)--(3,0.6);
\draw (3,0.6)--(3.3,0.6);
\draw (3.3,0.6)--(3.3,0.3);
\draw (3.3,0.3)--(3,0.3); 

\draw (3,0)--(3,0.3);
\draw (3,0.3)--(3.3,0.3);
\draw (3.3,0.3)--(3.3,0);
\draw (3.3,0)--(3,0); 

\node at (3.1,-1.4) {$a$}; 
\draw (3,-0.6)--(3, -0.9);
\draw (3,-0.9)--(3.3,-0.9);
\draw (3.3,-0.9)--(3.3,-0.6);
\draw (3.3,-0.6)--(3,-0.6);

\draw (3.3,-0.6)--(3.6,-0.6);
\draw (3.6,-0.6)--(3.6,-0.9);
\draw (3.6,-0.9)--(3.3,-0.9);

\draw (3,-0.9)--(3,-1.2);
\draw (3,-1.2)--(3.3,-1.2);
\draw (3.3,-1.2)--(3.3,-0.9);

\node at (5.3,-0.3) {$M_2 = $};
\node at (6.15, 1.1) {$a$}; 
 \draw[<-]  (6.15, -0.1)--(6.15,-0.5);
\draw (6,0.6)--(6,0.9);
\draw (6,0.9)--(6.3,0.9);
\draw (6.3,0.9)--(6.3,0.6);
\draw (6.3,0.6)--(6,0.6); 

\draw (6,0.3)--(6,0.6);
\draw (6,0.6)--(6.3,0.6);
\draw (6.3,0.6)--(6.3,0.3);
\draw (6.3,0.3)--(6,0.3); 

\draw (6,0)--(6,0.3);
\draw (6,0.3)--(6.3,0.3);
\draw (6.3,0.3)--(6.3,0);
\draw (6.3,0)--(6,0); 

\node at (6.1,-1.4) {$a$}; 
\draw (6,-0.6)--(6, -0.9);
\draw (6,-0.9)--(6.3,-0.9);
\draw (6.3,-0.9)--(6.3,-0.6);
\draw (6.3,-0.6)--(6,-0.6);

\draw (6.3,-0.6)--(6.6,-0.6);
\draw (6.6,-0.6)--(6.6,-0.9);
\draw (6.6,-0.9)--(6.3,-0.9);

\draw (6,-0.9)--(6,-1.2);
\draw (6,-1.2)--(6.3,-1.2);
\draw (6.3,-1.2)--(6.3,-0.9);

\end{tikzpicture} 
\caption{The three different limits of $M$. }
\end{figure}

\end{example}

\medskip

{\bf Acknowledgments.}
We would like to thank A. Moura for telling us about the problem and useful discussions.

EM is partially supported by the Simons foundation grant number \#709444.

MF acknowledges the support from the FAPESP's project 2019/23380-0 during his studies in Campinas, where this project has been initiated.

\end{document}